\documentclass{amsart}
\usepackage{mathabx,mathtools,amsrefs,mathrsfs,math,enumerate,multirow}
\usepackage{tikz}
\usetikzlibrary{arrows,automata,calc}
\numberwithin{equation}{section}

\newenvironment{fsa}[1][auto]{\begin{tikzpicture}[->,>=stealth',
    shorten >=1pt,auto,node distance=3cm,double distance between line centers=0.45ex,
    initial text=,accepting/.style=accepting by arrow,
    every state/.style={inner sep=3pt,minimum size=0pt,fill=gray,text=white,draw=none},
    every loop/.style={looseness=12},semithick,#1]}{\end{tikzpicture}}

\newenvironment{dualmoore}[1][auto]{\begin{tikzpicture}[->,>=stealth',
    shorten >=1pt,auto,node distance=2cm,double distance between line centers=0.45ex,
    initial text=,accepting/.style=accepting by arrow,
    every state/.style={rectangle,inner sep=3pt,minimum size=0pt},
    every loop/.style={looseness=12},semithick,#1]}{\end{tikzpicture}}
 
\makeatletter
\newcommand*{\rom}[1]{\expandafter\@slowromancap\romannumeral #1@}
\makeatother

\newcommand{\MM}{\mathcal{M}}
\newcommand{\MN}{\mathcal{N}}

\usepackage{amsmath}
\usepackage{amssymb}
\usepackage{amsthm}
\usepackage{enumitem}
\usepackage{todonotes}
\usepackage{hyperref}
\usepackage{url}

\setlist[enumerate]{label=\rm{(\arabic*)}, ref=(\arabic*)}

\DeclareMathOperator{\St}{St}
\DeclareMathOperator{\Aut}{Aut}

\DeclareMathOperator{\Sym}{Sym}

\DeclareMathOperator{\Deg}{Degree}
\DeclareMathOperator{\Lan}{Lan}

\newcommand{\ZS}{\bowtie}

\newcommand{\NN}{\mathbb{N}}
\newcommand{\ZZ}{\mathbb{Z}}

\newtheorem{lemma}[thm]{Lemma}

\theoremstyle{definition}

\newtheorem*{notation}{Notation}
\newtheorem{example}[thm]{Example}
\newtheorem*{remark}{Remark}

\title{On the existence of free subsemigroups in reversible automata semigroups}

\author{Dominik Francoeur and Ivan Mitrofanov}
\thanks{The first author was supported by a Doc.Mobility grant from the Swiss National Science Foundation.
The second author was supported by the “@raction” grant ANR-14-ACHN-0018-01.}

\begin{document}

\begin{abstract}
We prove that the semigroup generated by a reversible Mealy automaton contains a free subsemigroup of rank two if and only if it contains an element of infinite order.
\end{abstract}

\maketitle

\section{Introduction}

Groups and semigroups defined by Mealy automata have attracted considerable attention since their introduction.
One of the reasons for this interest is that the apparent simplicity of their definition belies the complex behaviours that they can exhibit.
Indeed, among them, one can find infinite finitely generated torsion groups \cite{Gri80}, groups and semigroups of intermediate growth and amenable but not elementary amenable groups \cite{Gri83}.

It is natural to ask how the properties of a Mealy automaton can influence the algebraic behaviour of the group or semigroup that it generates.
In this paper, we investigate the existence of free subsemigroups of rank two in semigroups generated by a reversible Mealy automaton. More precisely, we prove the following theorem.
\begin{thm}\label{thm:MainTheorem}
Let $\MM$ be a reversible Mealy automaton and let $P_{\MM}$ be the semigroup generated by $\MM$. Then, $P_{\MM}$ contains a nonabelian free subsemigroup if and only if $P_{\MM}$ contains an element of infinite order.
\end{thm}

As a corollary of Theorem \ref{thm:MainTheorem}, we obtain a generalization of a result of Klimann \cite{Kli17}, who proved that a group generated by a bireversible automaton is of exponential growth as soon as it contains an element of infinite order.
\begin{cor}\label{cor:ExponentialGrowth}
If $G$ is a group generated by an invertible and reversible Mealy automaton, then $G$ is of exponential growth as soon as it contains an element of infinite order. In particular, no infinite virtually nilpotent group can be generated by an invertible and reversible Mealy automaton.
\end{cor}

\subsection*{Organisation of the paper}
In Section \ref{sec:Preliminaries}, we define Mealy automata and their duals, as well as several semigroups associated with them.
We use the notion of Zappa-Sz\'ep product of semigroups to help us describe these different semigroups in a uniform fashion.

The proof of Theorem \ref{thm:MainTheorem} relies on the interplay between the semigroup generated by a reversible automaton and the group generated by its dual automaton.
We investigate this connection in Section \ref{sec:OrbitsPeriod}.
More precisely, given a reversible Mealy automaton $\MM$ and a word $u$ in the set of states of $\MM$, we show in Lemma \ref{lemma:SFT} that the orbits of the powers of $u$ under the action of the group generated by the dual of $\MM$ form a regular language.
Furthermore, in Lemma \ref{lemma:ClassesFiniteIFFOrbitsFinite}, we show that $u$ represents an element of infinite order in the semigroup generated by $\MM$ if and only if this language is not uniformly bounded.
%For an invertible Mealy automaton $\MM = (Q, A, \tau)$ and a word $u\in A^*$ we consider 
%the orbit $P_{\MM} \cdot u^{\omega}$ of a periodic sequence $u^{\omega}$. 
%Lemma \ref{lemma:ClassesFiniteIFFOrbitsFinite} say that if $u$ has infinite order 
%then this orbit should be infinite, 
%and Lemma \ref{lemma:SFT} say that the set of prefixes of sequences in this orbit  forms a regular language.

In Section \ref{sec:FreeSubsemigroups}, we use this regular language in the case where it is not uniformly bounded to find a free subsemigroup of rank two in the semigroup generated by a reversible automaton, thus proving Theorem \ref{thm:MainTheorem}.
%In Section \ref{sec:FreeSubsemigroups} we find in this regular language a free subsemigroup of rank two 
%and prove Theorem \ref{thm:MainTheorem}.

In Section \ref{section:languages}, we investigate in more details the orbits of automata groups and semigroups.
We first study potential generalizations of Lemma \ref{lemma:SFT}. 
We show that if we consider a preperiodic sequence instead of periodic one, a non-reversible Mealy automaton instead of a reversible one, or a subgroup instead the whole group, then the language thus obtained might not be regular.
However, for bi-reversible automata, we prove that Lemma \ref{lemma:SFT} hold even if we consider preperiodic sequences instead of periodic ones.

We also study the existence of infinite orbits of periodic sequences.
We show that there exists an automaton group containing an infinite subgroup such that the orbit of every periodic sequence under the action of this subgroup is finite.
However, at the moment, we do not know whether there exists an infinite automaton group such that every periodic sequence has a finite orbit.
We discuss what is known about this question in Section \ref{sec:Outlook}.

\subsection*{Acknowledgements}
The authors would like thank Laurent Bartholdi for suggesting this problem and for many useful discussions.
%The first author was supported by a Doc.Mobility grant from the Swiss National Science Foundation.
%The second author was supported by the “@raction” grant ANR-14-ACHN-0018-01.

\section{Preliminaries}\label{sec:Preliminaries}

In this section, we will review some facts about Mealy automata and the semigroups 
or groups that they generate.

\subsection{Alphabets, words and sequences}
Let $A$ be a finite set. We will denote by $A^*$ the free monoid on $A$.
In other words, $A^*$ is the set of words in the alphabet $A$, including the empty word $\epsilon$, equipped with the operation of concatenation. 
In what follows, we will make no distinction in our notation between the free monoid $A^*$ and its underlying set of words.

We will denote by $A^\omega$ the set of right-infinite words in the alphabet $A$. Thus, elements of $A^\omega$ are functions from $\N$ to $A$. We will call such element \emph{sequences}. There is a well-defined operation of concatenation on the left between an element $u\in A^*$ and an element $\xi\in A^\omega$, whose result is an element of $A^\omega$ that we will denote simply by $u\xi\in A^\omega$.

Given some $u=u_0u_1\dots u_n\in A^*$, we will denote by $u^\omega \in A^\omega$ the sequence 
\[u^\omega=u_0u_1\dots u_n u_0 u_1 \dots u_n \dots.\]

For $u=u_0u_1\dots u_n\in A^*$, we call a \emph{prefix} of $u$ any word of the form $v=u_0u_1\dots u_m$ with $m\leq n$ and a \emph{subword} of $u$ any word of the for $w=u_iu_i+1\dots u_i+j$ with $1\leq i \leq i+j \leq n$. In a similar fashion, we define the notions of prefix, subword and subsequences of elements of $A^\omega$.

\subsection{Mealy automata}
\begin{defn}
A \emph{Mealy automaton} is a tuple $\MM=(Q, A, \tau)$, where $Q$ and $A$ are finite sets 
called respectively the \emph{set of states} and the \emph{alphabet}, 
and $\tau\colon Q\times A \rightarrow A \times Q$ is a map called the \emph{transition map}.
\end{defn}

Please note that we will sometimes omit the word "Mealy", but unless otherwise specified, in the rest of this text, by "automaton" we will mean "Mealy automaton".

\begin{notation}
Let $\MM=(Q, A, \tau)$ be a Mealy automaton. For $q\in Q$ and $a\in A$, 
we will write $\tau(q,a) = (q\cdot a, q@a)$.
\end{notation}

A Mealy automaton $\MM=(Q, A, \tau)$ can be represented by a labelled directed graph 
called its \emph{Moore diagram} (see Figure \ref{figure:MooreDiagram} for an example). 
The set of vertices of this graph is $Q$, 
and there is an edge from $p\in Q$ to $q\in Q$ labelled by $a|b$ if and only if 
there exists $a,b\in A$ such that $\tau(p,a)=(b,q)$.

\begin{figure}[h]
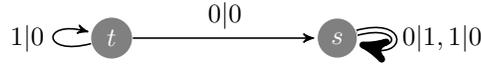

\centering
\[\begin{fsa}[baseline]
    \node[state] (t) at (0,0) {$t$};
    \node[state] (s) at (3,0) {$s$};
    \path (t) edge node {$0|0$} (s)
    	  (t) edge [loop left] node {$1|0$} (t)
          (s) edge [double,loop right] node {$0|1,1|0$} (s);
  \end{fsa}
\]
\caption{The Moore diagram of a Mealy automaton.}\label{figure:MooreDiagram}
\end{figure}

Given an automaton $\MM$ and a state $q\in Q$, we can define a map 
$q\cdot \colon A \rightarrow A$ by $(q\cdot)(a) = q\cdot a$ for all $a\in A$. 
Likewise, for each letter $a\in A$, we can define a map $@a\colon Q \rightarrow Q$.

\begin{defn}
A Mealy automaton $\MM=(Q,A,\tau)$ is said to be \emph{invertible} if 
for all $q\in Q$, the map $q\cdot \colon A \rightarrow A$ is a bijection. 
It is said to be \emph{reversible} if for all $a\in A$, the map $@a\colon Q\rightarrow Q$ is a bijection. 
Finally, it is said to be \emph{bi-reversible} if it is invertible, reversible and the map
 $\tau$ is a bijection.
\end{defn}

\begin{example}
The automaton given by $\MM = (\ZZ/2\ZZ, \ZZ/2\ZZ, \tau(q,a) = (q+a, q+a))$ is invertible, reversible, 
but not bi-reversible.
\end{example}

The definition of a Mealy automaton is symmetric with respects to the set $Q$ and $A$. 
Thus, we can interchange them to obtain a new automaton called the \emph{dual automaton}.

\begin{defn}
Let $\MM=(Q,A,\tau)$ be a Mealy automaton. The \emph{dual} of $\MM$ is the automaton 
$\partial \MM=(A,Q,\tau')$, where $\tau'\colon A\times Q \rightarrow Q\times A$ 
is given by $\tau'(a,q) = (q@a, q\cdot a)$.
\end{defn}

\begin{rem}
Let $\MM$ be a Mealy automaton. Then, $\partial (\partial \MM) = \MM$. 
Furthermore, $\MM$ is reversible if and only if $\partial \MM$ is invertible.
\end{rem}

\subsection{Zappa-Sz\'ep product of monoids}

\begin{defn}
Let $S$ be a monoid and let $X$ and $Y$ be submonoids of $S$ such that 
any $s\in S$ can be written in a unique way  
as $s =xy$ with $x \in X$ and $y \in Y$. 
Then $S$ is said to be an {\it internal Zappa-Sz\'ep product} of $X$ and $Y$, and we write
$S = X \ZS Y$.
\end{defn}

If $S = X \ZS Y$, then for any $x \in X$, $y = Y$ we must have unique elements 
$x'\in X$, $y'\in Y$ so that
$yx = x'y'$. 
So, we have two functions $Y\times X \to X$, $(y,x) \longmapsto y\cdot x$ and 
$Y \times X \to Y$, $(y,x) \longmapsto y@x$.

In $ X \ZS Y$ it's easy to prove properties of $\cdot$ and $@$:

\begin{enumerate}
\item  $y_1y_2 \cdot x = y_1\cdot(y_2\cdot x)$ for any $y_1,y_2 \in Y$ and $x\in X$;
\item $y @ x_1x_2 = (y@x_1)@x_2$ for any $y \in Y$ and $x_1,x_2 \in X$;
\item $y \cdot x_1x_2 = (y\cdot x_1) (y@x_1)\cdot x_2$ for any $y \in Y$ and $x_1,x_2 \in X$;
\item $y_1y_2 @ x = y_1@(y_2\cdot x) y_2@x$ for any $y_1,y_2 \in Y$ and $x\in X$.
\end{enumerate}

On the other hand, given two semigroups $X$ and $Y$ with maps $Y\times X \to X$ and $Y \times X \to Y$ which satisfy properties (1) -- (4), we can construct their Zappa-Sz\'ep product \cite{Zappa42}.

The property $(1)$ means that $\cdot$ is a left action of $Y$ on $X$, 
the property $(2)$ means that $@$ is a right action of $X$ on $Y$.

The actions of $X$ on $Y$ and of $Y$ on $X$ are not necessarily faithful. 
We define equivalence relations $\stackrel{@}{\sim}$ on $X$ and $\stackrel{\cdot}{\sim}$ on $Y$ as follows:

\begin{align*}
\text{for } x_1, x_2 \in X: x_1 \stackrel{@}{\sim} x_2 \text{ if and only if } y@x_1 = y@x_2 \text{ for all } y\in Y; \\
\text{for } y_1, y_2 \in Y: y_1 \stackrel{\cdot}{\sim} y_2 
\text{ if and only if } y_1\cdot x = y_2\cdot x \text{ for all } x\in X. \\
\end{align*}

It's clear that $\stackrel{\cdot}{\sim}$ and $\stackrel{@}{\sim}$ are congruence relations. 
Therefore, we can construct new monoids $X/{\stackrel{@}{\sim}}$ and $Y/{\stackrel{\cdot}{\sim}}$. We denote the actions of $X/{\stackrel{@}{\sim}}$ on $Y$ and of $Y/{\stackrel{\cdot}{\sim}}$ on $X$ by the same symbols $@$ and $\cdot$. 
For $x\in X$, we denote its congruence class by $[x]_{\stackrel{@}{\sim}}\in X/{\stackrel{@}{\sim}}$ and similarly, for $y\in Y$, we denote its congruence class by $[y]_{\stackrel{\cdot}{\sim}}\in Y/{\stackrel{\cdot}{\sim}}$.

In general the action of $X$ on $Y$ does not preserve congruence classes 
(see Example \ref{example:ClassesNotPreserved}), 
but we can prove something for semigroups with cancellation properties.

\begin{lemma}\label{lemma:ActionOnLeft}
Consider $S = X \ZS Y$ and let $Y$ have the right cancellation property. 
Suppose $x_1, x_2 \in X$ and $x_1\stackrel{@}{\sim} x_2$. 
Then for any $y\in Y$ we have $y\cdot x_1 \stackrel{@}{\sim} y\cdot x_2$.
\end{lemma}

\begin{proof}
We want to show that $s@(y \cdot x_1) = s@(y \cdot x_2)$ for any $s \in Y$.
We know that $sy @ x_1 = sy @ x_2$. Hence,
\begin{align*}
s@(y \cdot x_1)y@x_1 = sy@x_1 = sy@x_2 = s@(y \cdot x_2)y@x_2; 
\end{align*}
Since $y@x_1 = y@x_2$ and $Y$ has the right cancellation property, 
we conclude that $s@(y \cdot x_1) = s@(y \cdot x_2)$.
\end{proof}

The proof of the next lemma is similar:

\begin{lemma}\label{lemma:ActionOnRight}
Consider $S = X \ZS Y$ and let $Y$ have the left cancellation property. 
Suppose $y_1, y_2 \in Y$ and $y_1\stackrel{\cdot}{\sim} y_2$. 
Then for any $x\in X$ we have $y_1@x \stackrel{\cdot}{\sim} y_2@ x$.
\end{lemma}

If the action $@$ preserves congruence classes of $Y$, 
we denote the action of $X$ on $Y/{\stackrel{\cdot}{\sim}}$ by the same symbol $@$.

\begin{prop}
Consider $S = X\ZS Y$ and let the right action of $X$ preserve $\stackrel{\cdot}{\sim}-$congruence classes on $Y$.
Let $\stackrel{Y}{\sim}$ be the relation on $S$ defined by $s_1 \stackrel{Y}{\sim} s_2$ if $s_1 = xy_1$, $s_2 = xy_2$ 
and $y_1 \stackrel{\cdot}{\sim} y_2$, where $x \in X$ and $y_i \in Y$.
Then $\stackrel{Y}{\sim}$ is a congruence relation, and $S/{\stackrel{Y}{\sim}} \cong X \ZS (Y/{\stackrel{\cdot}{\sim}})$.
\end{prop}

\begin{proof}
It is clear from the fact that $\stackrel{\cdot}{\sim}$ is an equivalence relations that 
$\stackrel{Y}{\sim}$ is an equivalence relation. 
It remains to show that for all $s_1,s_2, s_3 \in S$, if $s_1\stackrel{Y}{\sim} s_2$, then 
$s_3s_1\stackrel{Y}{\sim} s_3s_2$ and 
$s_1s_3\stackrel{Y}{\sim} s_2s_3$.
Let $x,x_3 \in X$ and $y_1,y_2,y_3 \in Y$ be such that $s_1=xy_1$, $s_2=xy_2$ and $s_3 = x_3y_3$, 
and let us suppose that $y_1\stackrel{\cdot}{\sim} y_2$. We have
\[s_3s_1 = x_3y_3xy_1 = x_3(y_3\cdot x) (y_3 @x)y_1;\]
\[s_3s_2 = x_3y_3xy_2 = x_3(y_3\cdot x) (y_3 @x)y_2;\]
and $(y_3 @x)y_1 \stackrel{\cdot}{\sim} (y_3 @x)y_2.$

We also have 
\[s_1s_3 = xy_1x_3y_3 = x(y_1 \cdot x_3) (y_1@x_3)y_3;
\]
\[s_2s_3 = xy_2x_3y_3 = x(y_2 \cdot x_3) (y_2@x_3)y_3;\]
$y_1 \cdot x_3 = y_2 \cdot x_3$ and (since the action $@$ preserves congruence classes) 
$y_1@x_3\stackrel{\cdot}{\sim}y_2@x_3$.

Any element $s\in S/{\stackrel{Y}{\sim}}$ can be represented as $s = xy$ in a unique way, 
where $x\in X$, $y \in Y/{\stackrel{\cdot}{\sim}}$.
\end{proof}

Similarly, we have

\begin{prop}
Consider $S = X\ZS Y$ and let the left action of $Y$ preserve $\stackrel{@}{\sim}-$congruence classes on $X$.
Let $\stackrel{X}{\sim}$ be the relation on $S$ defined by $s_1 \stackrel{X}{\sim} s_2$ if $s_1 = x_1y$, $s_2 = x_2y$ 
and $x_1 \stackrel{@}{\sim} x_2$, where $x_i \in X$ and $y \in Y$.
Then $\stackrel{X}{\sim}$ is a congruence relation, and $S/{\stackrel{X}{\sim}} \cong (X/{\stackrel{@}{\sim}}) \ZS Y$.
\end{prop}

\subsection{Monoids associated with Mealy automata}

\begin{defn}
Let $\MM=(Q,A,\tau)$ be a Mealy automaton. We will denote by $\Delta_{\MM}$ 
the semigroup defined by the following presentation:
\[\Delta_{\MM} = \langle Q, A \mid qa = (q\cdot a)(q@a) \rangle.\]
\end{defn}

\begin{rem}\label{rem:AntiIsomorphicFundamentalSemigroups}
As is readily seen from the presentation, the map 
$\alpha \colon \Delta_{\MM} \rightarrow \Delta_{\partial \MM}$ 
defined by $\alpha(x_1x_2\dots x_n) = x_n \dots x_2 x_1$ 
for $x_1,x_2,\dots, x_n \in Q\cup A$ is a well-defined anti-isomorphism. Thus, $\Delta_{\MM}$ and 
$\Delta_{\partial \MM}$ might not be isomorphic in general, but they are anti-isomorphic.
\end{rem}

\begin{prop}\label{prop:NormalForm}
Let $\MM=(Q,A,\tau)$ be a Mealy automaton. 
Then $\Delta_{\MM}$ is an internal Zappa-Sz\'ep 
product of the free monoids $A^*$ and $Q^*$, i.e. for any $x\in \Delta_{\MM}$, 
there exists a unique choice of elements 
$a_1, a_2, \dots, a_k\in A$ and $q_1, q_2, \dots, q_l\in Q$ 
such that $x=(a_1a_2\dots a_k)(q_1q_2\dots q_l)$.
\end{prop}
\begin{proof}
Let $\left(Q\cup A\right)^*$ be the free monoid on $Q\cup A$ 
and let $\iota\colon \left(Q\cup A\right)^* \rightarrow \N_0$ be the map defined by
\[\iota(x_1x_2\dots x_n) = \left|\{(i,j)\mid x_i\in Q, x_j\in A, 1\leq i < j \leq n\}\right|\]
where $x_1,x_2, \dots, x_n\in Q\cup A$. 
It is clear that $\iota(x_1x_2\dots x_n) = 0$ if and only if 
there exists $k\in \N$ such that $x_i\in A$ if $1\leq i \leq k$ and $x_i\in Q$ if $k<i\leq n$.

We define a binary relation $\rightarrow$ on $\left(Q\cup A\right)^*$ by
\[x_1\dots x_{i-1}qax_{i+2} \dots x_n \rightarrow x_1\dots x_{i-1}(q\cdot a)(q@a)x_{i+2}\dots x_n\]
where $x_1,\dots, x_n \in Q\cup A$, $q\in Q$ and $a\in A$.

One can easily see that if $x\rightarrow y$, then $\iota(y) = \iota(x)-1$. 
Therefore, as $\iota(x)\geq 0$ for all $x\in \left(Q\cup A\right)^*$, 
there can exist no infinite chain $x_1\rightarrow x_2 \rightarrow x_3 \rightarrow \dots$. 
Furthermore, if $x, y, z \in \left(Q\cup A\right)^*$ are such that 
$x\rightarrow y$ and $x\rightarrow z$, then there exists $w\in \left(Q\cup A\right)^*$ 
such that $y\rightarrow w$ and $z\rightarrow w$. 
Indeed, let us write $x=x_1x_2 \dots x_n$, $y=y_1y_2\dots y_n$ and $z=z_1z_2\dots z_n$ 
(notice that it follows from the definition that $y$ and $z$ must have the same length as $x$). 
Then, there exist $i,j\in \N$ such that 
$y_k=x_k$ for all $k\neq i, i+1$ and $z_k=x_k$ for all $k\neq j, j+1$. 
If $i=j$, then $y=z$ and there is nothing to prove. 
If $i\neq j$, then $i\neq j+1$ and $j \neq i+1$, since $x_i,x_j \in Q$ and $x_{i+1}, x_{j+1} \in A$. 
Thus, we can define $w=w_1w_2 \dots w_n$, where
\[w_k = \begin{cases}
y_k & \text{ if } k=i, i+1 \\
z_k & \text{ otherwise}
\end{cases}\]
and we get that $y\rightarrow w$ and $z\rightarrow w$.

Consequently, we can apply the Diamond lemma to conclude that 
if two words $x,y \in \left(Q\cup A\right)^*$ represent the same element in $\Delta_{\MM}$, 
then there exists a unique word $z\in \left(Q\cup A\right)^*$ representing the same element 
and satisfying $\iota(z)=0$. The result follows.
\end{proof}
\begin{cor}
The homomorphisms $i_{Q} \colon Q^* \rightarrow \Delta_{\MM}$ 
and $i_{A}\colon A^*\rightarrow \Delta_{\MM}$ are injective.
\end{cor}

This monoid $\Delta_{\MM} = A^*\ZS Q^*$ gives us a left action $\cdot$ of $Q^*$ on $A^*$ 
and a right action $@$ of $A^*$ on $Q^*$.

\begin{remark}
We denote by $\cdot$ and $@$ both functions in Zappa-Sz\'ep product and operations of a Mealy automaton, 
but since these functions give the same results, there should be no risk of confusion.
\end{remark}
  
Note that the actions $\cdot: Q^* \times A^*\to A^*$ and $@:Q^* \times A^*$ preserve lengths of words.
The action of $Q$ can be  described by the Moore diagram of $\MM$. 
The action of $q\in Q$ on $A^*$ is determined as follows: given $a_1a_2\dots a_n \in A^*$ and $q\in Q$, 
find in the Moore diagram the unique path starting at $q$ and whose first label letters read 
$a_1\dots a_2$, let $b_1\dots b_n$ be the second label letters; then $q \cdot a_1\dots a_n = b_1 \dots b_n$. 

For a Mealy automaton $\MM = (Q,A,\tau)$ we denote the relation $\stackrel{\cdot}{\sim}$ on $Q^*$ by $\stackrel{Q}{\sim}$, the relation $\stackrel{@}{\sim}$ on $A^*$ by $\stackrel{A}{\sim}$, the monoid $Q^*{\stackrel{\cdot}{\sim}}$ by $P_{\MM}$ and the monoid $A^*/{\stackrel{@}{\sim}}$ by $D_{\MM}$.

\begin{defn}
The semigroup $P_{\MM}$ is called the \emph{automaton semigroup} of $\MM$. 
The semigroup $D_{\MM}$ is called the \emph{dual automaton semigroup} of $\MM$.
\end{defn}

The monoid $P_{\MM} = Q^*/{\stackrel{Q}{\sim}}$ is generated by $[Q]_{\stackrel{Q}{\sim}}$ and the semigroup $D_{\MM}$ is generated by $[A]_{\stackrel{A}{\sim}}$, where $[\cdot]_{\stackrel{Q}{\sim}}: Q^* \to P_{\MM}$ and $[\cdot]_{\stackrel{A}{\sim}}: A^* \to D_{\MM}$ are the canonical maps.

%\begin{thm}\label{thm:main}
%Let $\MM = (Q,A,\tau)$ be an invertible Mealy automaton and let $x\in D_{\MM}$ 
%have infinite order. 
%Then there exist elements $y, z \in D_{\MM}$ 
%such that they generate a non-commutative free subsemigroup.
%\end{thm}

\begin{rem}\label{rem:MainThmEquivalentMainDual}
Under our definitions, the dual automaton semigroup $D_{\MM}$ of $\MM$ 
is in general not isomorphic to the automaton semigroup $P_{\partial \MM}$ of the dual automaton $\partial \MM$. 
However, it follows from Remark \ref{rem:AntiIsomorphicFundamentalSemigroups} 
that they are canonically anti-isomorphic. Thus, any property preserved by anti-isomorphisms 
(such as finiteness or the existence of a free subsemigroup) will be true in $P_{\partial \MM}$ 
if and only if it is true in $D_{\MM}$. 
Therefore, to prove Theorem \ref{thm:MainTheorem}, it is sufficient to prove the following dual version:
\begin{thm}\label{thm:mainDual}
%Let $\MM = (Q,A,\tau)$ be a reversible Mealy automaton and let $x\in P_{\MM}$ have
%infinite order. Then there exist elements $y, z \in P_{\MM}$ that generate a non-commutative free subsemigroup.\todo{Refer to theorem in intro.}
Let $\MM = (Q,A,\tau)$ be an invertible Mealy automaton.
Then, there exist elements $y, z \in D_{\MM}$ that freely generate a non-commutative free subsemigroup if and only if there exists $x\in D_{\MM}$ with infinite order.
\end{thm}
\end{rem}

In fact, we will prove a stronger version of Theorem \ref{thm:mainDual}. To state it, however, we first need a lemma.

\begin{lemma}\label{lemma:ActionOnDelta}
Let $\MM = (Q,A,\tau)$ be a Mealy automaton. 
\\1) Let $v, w \in A^*$ be such that $v \stackrel{A}{\sim} w$ and let $t\in Q^*$. 
Then $t \cdot v \stackrel{A}{\sim} t \cdot w$.
\\2)  let $v, w \in Q^*$ be such that $v \stackrel{Q}{\sim} w$ and let $t\in A^*$. Then $v @ t \stackrel{Q}{\sim} w @ t$.
\end{lemma}

\begin{proof}
Since the monoids $Q^*$ and $A^*$ are cancellative semigroups, 
it follows from Lemmas \ref{lemma:ActionOnRight} and \ref{lemma:ActionOnLeft} that the actions of $Q^*$ on $D_{\MM}$ and of $A^*$ on $P_{\MM}$ are well-defined.
\end{proof}

Therefore, we can consider monoids
$A^* \ZS P_{\MM}$, $D_{\MM} \ZS Q^*$ and $D_{\MM} \ZS P_{\MM}$.

For $S = A^* \ZS P_{\MM}$ we consider the relation $\stackrel{@}{\sim}$ on $A^*$, which we will denote by $\stackrel{D}{\sim}$, and the monoid $D'_{\MM}:=A^*/{\stackrel{D}{\sim}}$. 

In other words, the relation $\stackrel{D}{\sim}$ on $A^*$ is defined by $u_1 \stackrel{D}{\sim} u_2$ if 
$x@u_1 \stackrel{Q}{\sim} x@u_2$ for any $x\in Q^*$.

Here is a stronger version of Theorem \ref{thm:mainDual}:

\begin{thm}\label{thm:main2}
Let $\MM = (Q,A,\tau)$ be an invertible Mealy automaton.
Then, there exist elements $y, z \in D'_{\MM}$ that freely generate a non-commutative free subsemigroup if and only if there exists $x\in D_{\MM}$ of infinite order.
\end{thm}

In Section \ref{sec:FreeSubsemigroups}, we will prove this theorem and thus prove Theorem \ref{thm:MainTheorem}.

\begin{remark}
In summary, for a Mealy automaton $\MM = (Q,A,\tau)$, we denote
by $@$:
\begin{enumerate}
\item the right part of the operation $\tau$; 
\item the action of $A^*$ on $Q^*$;
\item the action of $A^*$ on $P_{\MM}$;
\item the action of $D_{\MM}$ on $Q^*$;
\item the action of $D_{\MM}$ on $P_{\MM}$;
\item the action of $D'_{\MM}$ on $P_{\MM}$.
\end{enumerate}

Similarly, we denote by $\cdot$:
\begin{enumerate}
\item the left part of the operation $\tau$; 
\item the action of $Q^*$ on $A^*$;
\item the action of $Q^*$ on $D_{\MM}$;
\item the action of $P_{\MM}$ on $A^*$;
\item the action of $P_{\MM}$ on $D_{\MM}$.
\end{enumerate}

These operations commute with the corresponding projections: for example, if $s @ u = s'$ where $s, s' \in Q^*$ and $u \in A^*$,
then $[s]_{\stackrel{Q}{\sim}} @ [u]_{\stackrel{A}{\sim}} = [s']_{\stackrel{Q}{\sim}}$.

\end{remark}

\begin{example}\label{example:ClassesNotPreserved}

Consider a Mealy automaton $\MM = (Q,A,\tau)$ 
where $Q =\{a, b, c\}$ and alphabet
\[A:=\{x_1, y_1, x_2, y_2, z_1, z_2\}.\] 

Its map
$\tau\colon Q\times A \to A\times Q$ is defined by the following table.

\[\begin{array}{cr|ccc|}
& & \multicolumn{3}{c|}{q \in Q}\\
& & a & b & c  \\ \hline
\multirow{6}{*}{\rotatebox{90}{$a \in A$}} 
& x_1 & (x_2, a) & (x_2, a) & (x_2, a)  \\
& y_1 & (y_2, a ) & (y_2, a) & (y_2, a)  \\
& x_2 & (x_2, a) & (x_2, b) & (x_2, c) \\
& y_2 & (y_2, a) & (y_2, c) & (y_2, b) \\
& z_1 & (z_2, a) & (z_1, b) & (z_2, c) \\
& z_2 & (z_2, a) & (z_2, b) & (z_2, c) \\
\hline
\end{array}\]

Note that all words in $Q^*$ that contain at least one $a$ act on $A^*$ 
in the same way (replacing all lower indices by $2$). 

Then $x_1 \stackrel{D}{\sim} y_1$, because both of $x_1$ and $y_1$ map all non-empty words of $Q^*$ to elements
$\stackrel{Q}{\sim}-$equivalent to $a$.

$b\cdot z_1 \neq c\cdot z_1 \Rightarrow b \stackrel{Q}{\not\sim} c$. 
Since $b @ x_2 = b \stackrel{Q}{\not\sim} c = b @ y_2$, then $x_2\stackrel{D}{\not\sim} y_2$. Since $x_1 \stackrel{D}{\sim} y_1$ but $a\cdot x_1 \stackrel{D}{\not\sim} a\cdot y_1$, 
we conclude that the action of $P_{\MM}$ on $D'_{\MM}$ is not well-defined.
\end{example}

\subsection{Groups generated by invertible automata}
In the case where the automaton $\MM=(Q,A,\tau)$ is invertible, 
it is natural to consider not only the automaton semigroup $P_{\MM}$, but an automaton group, which we will define below.

\begin{defn}
Let $\MM=(Q,A,\tau)$ be a invertible Mealy automaton and let 
$Q^{-1}=\{q^{-1}\mid q\in Q\}$ be the set of formal inverses of $Q$. 
The \emph{enriched automaton} of $\MM$ is the automaton 
$\widetilde{\MM}=(Q\sqcup Q^{-1}, A, \widetilde{\tau})$, 
where $\widetilde{\tau}(q,a) = \tau(q,a)$ and 
$\widetilde{\tau}(q^{-1}, q\cdot a) = (a, (q@a)^{-1})$ for all $q\in Q$ and $a\in A$.
\end{defn}

\begin{prop}\label{prop:SemigroupOfAugmentedAutomatonIsGroup}
Let $\MM=(Q,A,\tau)$ be an invertible Mealy automaton. 
Then $P_{\widetilde{\MM}}$ is a group and for all $q\in Q$ 
the elements $q$ and $q^{-1}$ are inverse elements.
\end{prop}
\begin{proof}

We will prove that $q^{-1}q \cdot u = u$ for any $q \in Q$ and $u\in A^*$ by induction on $|u|$. 
For $|u| = 0$ this is obvious. 
Suppose that $u = au'$. Then,
\[
q^{-1}qau' = q^{-1}(q\cdot a)(q @ a) u' = a (q@a)^{-1}(q@a) u'.
\]
By the induction hypothesis, $(q@a)^{-1}(q@a)\cdot u' = u'$. Thus, $q^{-1}q\cdot u = u$. Similarly we show that $qq^{-1} \cdot u = u$ for all $u\in A^*$.

\end{proof}

\begin{cor}\label{cor:IfAlreadyGroupThenNoChange}
Let $\MM=(Q,A,\tau)$ be an invertible Mealy automaton such that $P_{\MM}$ is a group. Then, $P_{\MM} = P_{\widetilde{\MM}}$.
\end{cor}
\begin{proof}
It is clear from the definition of $\widetilde{\MM}$ that $P_{\MM} \leq P_{\widetilde{\MM}}$. On the other hand, it follows from Proposition \ref{prop:SemigroupOfAugmentedAutomatonIsGroup} that if $P_{\MM}$ is a group, then the generators of $P_{\widetilde{\MM}}$ are contained in $P_{\MM}$. We conclude that $P_{\MM} = P_{\widetilde{\MM}}$.
\end{proof}

\begin{defn}
Let $\MM=(Q,A,\tau)$ be an invertible Mealy automaton. The \emph{automaton group} of $\MM$ is the group $P_{\widetilde{\MM}}$.
\end{defn}

In what follows, we will be interested in the existence of elements of infinite order in the semigroup $P_{\MM}$ of a Mealy automaton $\MM$.
If this automaton is invertible, then it is equivalent to look for elements of infinite order in the automaton group $P_{\widetilde{\MM}}$, as we will see in the next proposition.

\begin{prop}\label{prop:InfiniteOrderInSemigroupSameAsGroup}
Let $\MM=(Q,A,\tau)$ be an invertible automaton, let $P_{\MM}$ be the automaton semigroup of $\MM$ and let $P_{\widetilde{\MM}}$ be the automaton group of $\MM$. Then, $P_{\MM}$ contains an element of infinite order if and only if $P_{\widetilde{\MM}}$ does.
\end{prop}
\begin{proof}
It is clear that if $P_{\MM}$ contains an element of infinite order, then so does $P_{\widetilde{\MM}}$.
To show the converse, let us suppose that every element of $P_{\MM}$ is of finite order.
Then, the inverse of any element of $P_{\MM}$ is also an element of $P_{\MM}$, which means that $P_{\MM}$ is a group.
Therefore, by Corollary \ref{cor:IfAlreadyGroupThenNoChange}, we get that $P_{\MM} = P_{\widetilde{\MM}}$, which means that $P_{\widetilde{\MM}}$ is a torsion group.
\end{proof}

As we will see below, the dual semigroup is unaffected by the passage to the enriched automaton.

\begin{lem}\label{lemma:ProjectionsOfInversesFormula}
Let $\MM=(Q,A,\tau)$ be an invertible Mealy automaton and $\widetilde{\MM}$ be its enriched automaton. 
Then, for any $q\in Q$ and $v\in A^*$, we have $q^{-1}@v = (q@(q^{-1}\cdot v))^{-1}$.
\end{lem}
\begin{proof}
By Proposition \ref{prop:SemigroupOfAugmentedAutomatonIsGroup}, it suffices to prove that 
$q^{-1}@(q\cdot v) = (q@v)^{-1}$. If $v$ is of length 1 (i.e. if $v\in A$), 
then we get by definition that $q^{-1}@(q\cdot v) = (q@v)^{-1}$.

Now, for $v\in A^*$ of length $n$, let us write $v=v_1v_2\dots v_n$, 
with $v_1,v_2, \dots, v_n \in A$. On the one hand, we have $(q^{-1}q)@v = (q^{-1}@(q\cdot v))(q@v)$. 
On the other hand,
\[(q^{-1}q)@(v_1\dots v_n) = ((q^{-1}q)@v_1)@(v_2\dots v_n) = ((q@v_1)^{-1}(q@v_1))@(v_2\dots v_n).\]
Thus, by induction, $(q^{-1}q)@v = (q@v)^{-1}(q@v)$. 
It follows that $q^{-1}@(q\cdot v) = (q@v)^{-1}$.
\end{proof}

\begin{lem}\label{lemma:SemigroupHasSameActionOnFiniteSets}
Let $\MM=(Q,A,\tau)$ be an invertible Mealy automaton and $V\subset A^*$ be a finite set. 
Then, for any $s\in (Q\sqcup Q^{-1})^*$, there exists $t\in Q^*$ such that 
$s\cdot v = t\cdot v$ for all $v\in V$.
\end{lem}
\begin{proof}
It suffices to show that for all $q^{-1} \in Q^{-1}$, there exists $t\in Q^*$ 
such that $q^{-1}\cdot v = t\cdot v$ for all $v\in V$. Since the action of $Q^*$ 
on $A^*$ preserves lengths and the set $V$ is finite, the size of the orbits under the action 
of $Q^*$ of elements of $V$ is uniformly bounded. 
This, coupled with the invertibility of $\MM$, implies that 
there exists some $k\in \N$ such that $q^{k+1}\cdot v = v$ for all $v\in V$. 
Therefore, $q^{k}\cdot v = q^{-1}\cdot v$ for all $v\in V$, which concludes the proof.
\end{proof}

\begin{prop}\label{prop:EnrichedDualIsDual}
Let $\MM=(Q,A,\tau)$ be an invertible automaton and $\widetilde{\MM}$ be its enriched automaton. 
Then, $D_{\MM} = D_{\widetilde{\MM}}$.
\end{prop}
\begin{proof}
By definition, $D_{\MM} = A^*/{\stackrel{A}{\sim}}$ and $D_{\widetilde{\MM}}=A^*/{\stackrel{A}{\sim}'}$, 
where $\stackrel{A}{\sim}$ and $\stackrel{A}{\sim}'$ are two congruence relations on $A^*$. 
We need to show that $\stackrel{A}{\sim} = \stackrel{A}{\sim}'$.

Consider $\Delta_{\MM}$ as a subsemigroup of $\Delta_{\widetilde{\MM}}$. 
Thus, if $v,w\in A^*$ are such that $s@v = s@w$ for all $s\in (Q\sqcup Q^{-1})^*$, 
then in particular, 
$s@v=s@w$ for all $s\in Q^*$. 
Therefore, if $v\stackrel{A}{\sim}' w$, then $v\stackrel{A}{\sim} w$.

On the other hand, suppose that $v\stackrel{A}{\sim} w$. Then, by definition, we must have 
$q@v=q@w$ for all $q\in Q$. Let us now consider $q^{-1} \in Q^{-1}$. 
By Lemma \ref{lemma:ProjectionsOfInversesFormula}, we have 
$q^{-1}@v = (q@(q^{-1}\cdot v))^{-1}$ and $q^{-1}@w = (q@(q^{-1}\cdot w))^{-1}$. 
Now, by Lemma \ref{lemma:SemigroupHasSameActionOnFiniteSets}, there exists $t\in Q^*$ such that $q^{-1}\cdot v = t\cdot v$ and $q^{-1}\cdot w = t\cdot w$. 
It follows from Proposition \ref{lemma:ActionOnDelta} that if $v\stackrel{A}{\sim} w$, 
then $t \cdot v \stackrel{A}{\sim} t\cdot w$, which means that $q^{-1}\cdot v \stackrel{A}{\sim} q^{-1}\cdot w$. 
Therefore, $q@(q^{-1}\cdot v) = q@(q^{-1}\cdot w)$, which implies that $q^{-1}@v = q^{-1}@w$.

Now, for $s = s_n\dots s_2s_1\in (Q\sqcup Q^{-1})^*$, 
we have $(s_n\dots s_2s_1)@v = ((s_n\dots s_2)@(s_1\cdot v))(s_1@v)$ 
and $(s_n\dots s_2s_1)@w = ((s_n\dots s_2)@(s_1\cdot w))(s_1@w)$. 
From the argument above, we have $s_1\cdot v \stackrel{A}{\sim} s_1 \cdot w$ and $s_1@v = s_1@w$. 
Thus, the result follows by induction.
\end{proof}

It follows from Proposition \ref{prop:EnrichedDualIsDual} that 
if we are interested in the dual semigroup of an invertible automaton $\MM$, 
we can assume without loss of generality that $P_{\MM}$ is in fact a group.

\begin{defn}
A Mealy automaton $\MM = (Q,A,\tau)$ is called {\it self-invertible} if $P_{\MM}$ is a group.
\end{defn}

\begin{lemma}\label{lemma:ActionOnDeltaPrime}
Let $\MM = (Q,A,\tau)$ be an invertible Mealy automaton and let $v, w \in D_{\MM}$ be such that $v \stackrel{D}{\sim} w$. 
Then, for all $t\in P_{\MM}$, we have $t \cdot v \stackrel{D}{\sim} t \cdot w$.
\end{lemma}

\begin{proof}
Since $P_{\MM}$ can be embedded into a group $P_{\widetilde{\MM}}$, 
the monoid $P_{\MM}$ is cancellative. 
The result follows from Lemma \ref{lemma:ActionOnRight}.
\end{proof}

This means that for an invertible automaton $\MM$, the action of 
$P_{\MM}$ on $D'_{\MM}$ is well-defined and 
we can also consider $D'_{\MM}\ZS P_{\MM}$.

\subsection{Transformation wreath products}

\begin{defn}\label{defn:wreath}
Let $\Gamma$ be a monoid acting from the left on a finite set $X$.
Let $S$ be a semigroup, then the {\it transformation wreath product} 
$W = S \wr_X \Gamma$ is defined as the semidirect product $\Gamma \ltimes S^X$.
\end{defn}

The elements of $S^X$ are functions  $f:X \to S$ with coordinatewise multiplication, 
the monoid $\Gamma$ acts on $S^X$ from the right as
$f^{\pi}(x) = f(\pi(x))$.
Elements in $W = S \wr_X \Gamma$ are recorded as pairs $(\pi, f)$, 
where $\pi$ is a function $X\to X$ from $\Gamma$ and $f \in S^X$.
Multiplication in $W$ is given by
\[(\pi_1, f_1)(\pi_2, f_2) = (\pi_1 \pi_2, f_1^{\pi_2}f_2).
\]

If $S$ and $\Gamma$ are groups, then $S \wr_X \Gamma$ is also a group.

\begin{prop}\label{prop:WreathRecursion}
Let $S = B \ZS A$ and let $X$ be a finite subset of $B$ such that $A\cdot X \subseteq X$.
The action of $A$ on $X$ gives us a homomorphism $p_n\colon A \rightarrow \End(X)$, 
where $\End(X)$ is the monoid of maps from $X$ to itself.
We also have a map $s\colon A \rightarrow A^X$ given by $s(a)(x) = a@x$ for $a\in A$ and $x\in X$.
Consider the map 
\begin{align*}
\varphi: A \to A \wr_{X} \End(X)\\
a \mapsto (p(a), s(a)).
\end{align*}

Then $\varphi$ is a homomorphism.
\end{prop}

\begin{proof}
We have
\[\varphi(a_1)\varphi(a_2) = (p(a_1), s(a_1))(p(a_2), s(a_2)) = 
(p(a_1)p(a_2),s_n(a_1)^{p_n(a_2)}s(a_2)).\]
As $p$ is a homomorphism, $p(a_1)p(a_2) = p(a_1a_2)$.
Furthermore, for any $x \in X$, we have 
\begin{align*}
\left(s(a_1)^{p(a_2)}s(a_2)\right)(x) &= \left(s(a_1)^{p(a_2)}\right)(x) s(a_2)(x)\\
&=(a_1@(p(a_2)(x)))(a_2@x)\\
&=(a_1@(a_2\cdot x))(a_2@x) \\
&=a_1a_2@x = s(a_1a_2)(x).
\end{align*}
Therefore, $s_n(a_1)^{p_n(a_2)}s(a_2) = s(a_1a_2)$. We conclude that $\varphi(a_1)\varphi(a_2) = \varphi(a_1a_2)$.
\end{proof}

The following lemma is useful in applications to automata semigroups.

\begin{lemma}\label{lemma:UnionOfBounded}
Let $S = B \ZS A$, where the monoid $A$ has a finite generating set $E$ 
such that $E@B \subseteq E$. 
Let $X_1, X_2, \dots$ be a sequence of finite subsets of $B$ such that 
$A \cdot X_i \subseteq X_i$ for any $X_i$, 
and let $|X_i| < M$ for some constant $M$ independent on $i$.
Then $|\bigcup_i [X_i]_{\stackrel{@}{\sim}}| < \infty$.
\end{lemma}

\begin{proof}

For each $X_i$ construct a homomorphism as in Proposition \ref{prop:WreathRecursion}:
\begin{align*}
\varphi_i: A \to A \wr_{X_i} \End(X_i)\\
a \mapsto (p_i(a), s_i(a)).
\end{align*}

Choosing an arbitrary bijection between $X_i$ and the set $Y_i:=\{1,2, \dots, |X_i|\}$ 
yields a homomorphism
\[\widetilde{\varphi_i}\colon A \wr_{Y_i} \End(Y_i).\]

Since $\widetilde{\varphi_i}$ is a homomorphism, 
it is uniquely determined by the image of the generating set $E$. 
Since $e@b\in E$ for any $e\in E$ and $b\in B$, we have that $s_i(E) \subseteq E^{X_i}$. 
It follows that
\[\widetilde{\varphi_i}(E) \subseteq \End(Y_i)\times E^{Y_i},\]
which is a finite set. Hence, for a given value of $|X_i|$, 
there are only finitely many different homomorphisms $\widetilde{\varphi_i}$.

Therefore, if $|X_i|\leq M$ for all $i\in \NN$, there must exist a finite set 
$\{i_1,i_2,\dots, i_m\}$ such that for any $i\in \NN$, there exists $k\in \NN$ 
such that $\widetilde{\varphi_i} = \widetilde{\varphi_{i_k}}$. 
In particular, for any 
$x\in X_i$, there exists a $x'\in X_{i_k}$ such that $a@x=g@x'$ for all $a\in A$, which means that 
$x \stackrel{@}{\sim} x'$ by definition. Therefore,
\[ 
\left|\bigcup_{n=0}^{\infty} [X_i]_{\stackrel{@}{\sim}}\right| \leq \sum_{k=1}^{m}|X_{i_k}|<\infty.\]
\end{proof}

The following result is well-known (see for example \cite{SV11}), but we include a proof here for completeness.

\begin{thm}\label{thm:BothSemigroupsAreFinite}
Let $\MM = (Q, A, \tau)$ be a Mealy automaton. Then $P_{\MM}$ is finite if and only if $D_{\MM}$ is finite.
\end{thm}
\begin{proof}
Suppose $|P_{\MM}| = M$. Then, for any $u \in A^*$, the cardinality of the set $Q^* \cdot u$ is not larger than $M$.

We enumerate all finite words of $A^*$: $A^* = {u_1, u_2, \dots}$.
Apply Lemma \ref{lemma:UnionOfBounded} for $S = A^* \ZS Q^*$, $E = Q$, $X_i = Q^* \cdot u_i$.
Then $|D_{\MM}| = |[A^*]_{\stackrel{@}{\sim}}|=|\bigcup_i [X_i]_{\stackrel{@}{\sim}}| < \infty$.

If $|D_{\MM}| < \infty$, we can apply Lemma \ref{lemma:UnionOfBounded} 
to the dual automaton $\partial\MM$ and prove that $P_{\MM}$ is finite.
   
\end{proof}

\subsection{Automorphisms of rooted trees}

The free monoid $A^*$ can be identified with the $|A|$-regular rooted tree $T_A$ as follows: 
the set of vertices is $A^*$ and two vertices $v,w\in A^*$ are joined by an edge if and only if 
there exists $a\in A$ such that $w=va$ or $v=wa$.
The \emph{root} of $T_A$ is the empty word $\varepsilon$. The words of $A^*$ of length $n$ form the 
\emph{$n$-th level} of $T_A$.

Consider the group $\Aut(T_A)$ acting on $T_A$ from the left by automorphisms fixing the root. 
At each level $\Aut(T_A)$ acts  by permutations.

A map $f: A^* \to A^*$ belongs to $\Aut(T_A)$ if $f$ preserves lengths of words 
and lengths of longest common prefixes. 

\begin{defn}
Let $A$ be a finite alphabet, and let $S = A^* \ZS X$. 
If $X$ is a group, the action of $X$ on $A^*$ is faithful and $X \cdot A = A$, this action of $X$ is called \emph{self-similar} and $X$ is called a \emph{self-similar group}.
\end{defn}

Self-similar actions on $A^*$ are rooted automorphisms of the tree $T_A$.

For example, for any self-invertible Mealy automata $\MM = (Q,A,\tau)$ the action of $P_{\MM}$ on $A^*$ is self-similar.

\begin{rem}\label{rem:ActionOnInfiniteSequences}
Since the action of $Q^*$ on $A^*$ preserves the tree structure of $A^*$, 
we can define the action of $Q^*$ on
the set of all right-infinite sequences $A^{\omega}$:
for any $s \in Q^*$ and $a_1a_2\ldots \in A^{\omega}$ the sequence 
$\{s \cdot a_1\ldots a_n\}$ has a limit $b_1b_2\ldots \in A^{\omega}$, 
and we write $s \cdot a_1a_2\ldots = b_1b_2\ldots $
\end{rem}

\section{Orbits of periodic sequences.}\label{sec:OrbitsPeriod}

Let $G \leqslant \Aut(T_A)$ be a group acting on $A^*$ from the left and let $x$ be an element of $A^*$ or an element of $A^{\omega}$.
The orbit of $x$ will be denoted by $G\cdot x$.
\subsection{Self-similar groups and regular languages.}

\begin{notation} 
Let $a\in A$ and $G \leqslant \Aut(T_A)$ be a subgroup of $\Aut(T_A)$.
We denote by $a^{\omega}$ the infinite sequence $aaa\dots$ and by $\Lan_G(a)$ the language that consists of all prefixes of sequences in $G\cdot a^{\omega}$.
\end{notation}

In other words, $\Lan_G(a) = \bigsqcup_k G\cdot a^k$.

\begin{lemma}\label{lemma:FactorLanguage}
%Let $\MM = (Q, A, \tau)$ be a self-invertible Mealy automaton, let $a\in A$ and let $v\in \Lan_{P_{\MM}}(a)$. 
Let $G$ act self-similarly on $A^*$, let $a\in A$ and let $v\in \Lan_{G}(a)$.
Then $\Lan_{G}(a)$ contains all the subwords of $v$.
\end{lemma}

\begin{proof}
It suffices to show that if $v_1v_2 \in \Lan_G(a)$, then $v_1, v_2 \in \Lan_{G}(a)$. 
Let us write $v = v_1v_2$ and let $g\in G$ be such that $v = g \cdot a^{|v_1|}a^{|v_2|}$.
Then, $v_1 = g \cdot a^{|v_1|}\in \Lan_{G}(a)$ and 
$v_2 = (g @ a^{|v_1|})\cdot a^{|v_2|}\in \Lan_{G}(a)$.
\end{proof}

For a given word $v \in \Lan_G(a)$ we say that the cardinality  
$\#\{a\in A: va \in \Lan_G(a)\}$ is the {\it degree} of $v$ in $\Lan_G(a)$
and denote this number by $\Deg_{\Lan_G(a)}(v)$ (or by $\Deg(v)$ in case 
it is clear what language is considered.)

\begin{lemma}\label{lemma:EqualDegrees}
Let $G$ be a group acting self-similarly on $A^*$, $a\in A$ and $v_1, v_2 \in \Lan_{G}(a)$ with $|v_1| = |v_2|$.
%Let $\MM = (Q, A, \tau)$ be a self-invertible Mealy automaton, $a\in A$ and $v_1, v_2 \in \Lan_{P_{\MM}}(a)$ with $|v_1| = |v_2|$.
Then $\Deg(v_1) = \Deg(v_2)$.
\end{lemma}

\begin{proof}
Let $\Deg(v_1) = d$. W.l.o.g. we consider that 
\[v_1a_1, v_1a_2, \dots, v_1a_{d} \in \Lan_{G}(a).\] 
There exists $g\in G$ such that $g \cdot v_1 = v_2$.
It follows that 
\[g \cdot v_1a_1 = v_2 ((g@v_2)\cdot a_1), v_2 ((g@v_2)\cdot a_2), 
\dots, v_2 ((g@v_2)\cdot a_d) \in \Lan_{G}(a),\] i.e. 
$\Deg(v_2) \geq d$. Similarly, $\Deg(v_1) \geq \Deg(v_2)$.

\end{proof}

\begin{lemma}\label{lemma:DegreesDecrease}
%Let $\MM = (Q, A, \tau)$ be a self-invertible Mealy automaton, $a\in A$ and $k_1 < k_2 \in \N$.
Let $G$ be a group with a self-similar action on $A^*$, $a\in A$ and $k_1 < k_2 \in \N$.
Then $\Deg_{\Lan_{G}(a)}(a^{k_1}) \geq \Deg_{\Lan_{G}(a)}(a^{k_2})$.
\end{lemma}

\begin{proof}
If $a^{k_2}a_i \in \Lan_{G}(a)$, then $a^{k_1}a_i \in \Lan_{G}(a)$ by Lemma \ref{lemma:FactorLanguage}, since 
$a^{k_1}a_i$ is a subword of $a^{k_2}a_i$.
\end{proof}

\begin{lemma}\label{lemma:SFT}
%Let $\MM = (Q, A, \tau)$ be a self-invertible Mealy automaton and $a\in A$. 
%Then, there exist $T, d \in \NN$ such that $\Deg_{\Lan_{P_{\MM}}}(v) = d$ for all 
%$v \in L_{P_{\MM}}(a)$ with $|v| \geqslant T$. Furthermore, for any word $v\in A^*$, the following statements are equivalent:
%\begin{enumerate}
%\item $v \in \Lan_{P_{\MM}}(a)$;\label{item:lemmaSFTvInLang}
%\item all subwords of $v$ of length at most $T+1$ belong to $\Lan_{P_{\MM}}(a)$.\label{item:lemmaSFTSubwordsInLang}
%\end{enumerate}
Let $G$ be a group acting self-similarly on $A^*$ and $a\in A$. 
Then, there exist $T, d \in \NN$ such that $\Deg_{\Lan_{G}}(v) = d$ for all 
$v \in L_{G}(a)$ with $|v| \geqslant T$. Furthermore, for any word $v\in A^*$, the following statements are equivalent:
\begin{enumerate}
\item $v \in \Lan_{G}(a)$;\label{item:lemmaSFTvInLang}
\item all subwords of $v$ of length at most $T+1$ belong to $\Lan_{G}(a)$.\label{item:lemmaSFTSubwordsInLang}
\end{enumerate}
\end{lemma}

\begin{proof}
Since the sequence of integers $\Deg(a), \Deg(a^2), \Deg (a^3), \dots$
is non-increasing then there exist $T, d \in \NN$ such that $\Deg(a^k) = d$ for $k \geq T$.
It follows from Lemma \ref{lemma:EqualDegrees} that if $|v| \geq T$ then $\Deg(v) = d$. This proves the first claim.

We have $\ref{item:lemmaSFTvInLang}\Rightarrow \ref{item:lemmaSFTSubwordsInLang}$ by Lemma \ref{lemma:FactorLanguage}.
To prove $\ref{item:lemmaSFTSubwordsInLang} \Rightarrow \ref{item:lemmaSFTvInLang}$, let us first notice that if there are $m$ words of length $T$ in $\Lan_{G}(a)$, then the previous part implies that there are $md^k$ words of length $T+k$ in $\Lan_{G}(a)$.

Consider the language 
\[\widetilde{\Lan}:=\{w \in A^*\mid u\in \Lan_{G}(a) \text{ for all subwords $u$ of $w$ of length at most $T+1$}.\}
\] 
It follows from Lemma \ref{lemma:FactorLanguage} that $\Lan_{G}(a) \subseteq \widetilde{\Lan}$. 
On the other hand, we can estimate  the number of words of each length in $\widetilde{\Lan}$.
It is clear that a word of length at most $T+1$ belongs to $\widetilde{\Lan}$ if and only if it belongs to $\Lan_{G}(a)$.
Thus, $\widetilde{\Lan}$ and $\Lan_{G}(a)$ contain the same number of elements of length at most $T+1$.

We will now show by induction that for all $k\in \N$, $\widetilde{\Lan}$ contains no more than $md^k$ words of length $T+k$. 
For $k = 1$, this is obvious.
Let us now suppose that for some $k\in \N$, there are $K$ words  of length $T+k$ in $\widetilde{\Lan}$, where $K\leq md^k$.
Then, as the language $\widetilde{\Lan}$ is clearly closed under the operation of taking prefixes, any word  $w\in \widetilde{\Lan}$ of length $T+k+1$ must start with one of these $K$ words. 
If there are more than  $dK$ words of length $T+k+1$, then by the pigeonhole principle, one of these $K$ words 
can be continued to the right in more than $d$ ways.
Thus, if $v$ is the word formed by the last $T$ letters of this word, we have that $\Deg(v)\geq d+1$, which is absurd. Therefore, we can conclude that $\widetilde{\Lan}$ and  $\Lan_{G}(a)$ contain the same number of elements of length $l$ for all $l\in \N$.

This shows that $\Lan_{G}(a) = \widetilde{\Lan}$.
\end{proof}

It's clear that  if $|G\cdot a^{\omega}| = \infty$ then $d > 1$.

\begin{rem}
Let $P_{\MM}$ be a semigroup generated by invertible Mealy automata 
and let $u^{\omega}$ be an arbitrary periodic sequence. 
The lemma \ref{lemma:SFT} gives us a way to describe the structure of 
$P_{\MM} \cdot u^{\omega}$: the prefixes of this set form a regular language.
In section \ref{section:languages} we try to generalize this observation.
\end{rem}

\subsection{Order of elements and finiteness of orbits.}

%Let $\MM=(Q,A,\tau)$ be an invertible automaton. Consider a right-infinite periodic sequence 
%$u^{\omega} \in A^{\omega}$. The orbit $P_{\MM}\cdot u^{\omega}$ can be finite or infinite. 

%On the other hand, the words $u, u^2, u^3, \dots$ are different in $A^*$, 
%but what can be said about elements $[u^k]_{\stackrel{A}{\sim}}\in D_{\MM}$ 
%or about elements $[u^k]_{\stackrel{D}{\sim}} \in D'_{\MM}$?

Let $\MM=(Q,A,\tau)$ be an invertible automaton. As we will see, for $u\in A^*$, the order of $[u]_{\stackrel{A}{\sim}}\in D_{\MM}$ is related to the size of the orbit of $u^{\omega} \in A^{\omega}$ under the action of $P_{\MM}$.

\begin{lemma}\label{lemma:ClassesFiniteIFFOrbitsFinite} 
Let $\MM=(Q,A,\tau)$ be a self-invertible automaton and let $u\in A^*$. 
Then the following statements are equivalent: 

\begin{enumerate}
\item $|P_{\MM}\cdot u^\omega|<\infty$;\label{item:FiniteOrbit}
\item $\left|[P_{\MM}\cdot u^n ]_{\stackrel{A}{\sim}}\right|$ is bounded independently of $n$;\label{item:BoundedAClass}
\item $\left|[P_{\MM}\cdot u^n ]_{\stackrel{D}{\sim}}\right|$ is bounded independently of $n$;\label{item:BoundedDClass}
\item $\left|[\bigcup_{n=0}^{\infty} P_{\MM}\cdot u^n ]_{\stackrel{A}{\sim}}\right| < \infty$;\label{item:FiniteAClass}
\item $|[\bigcup_{n=0}^{\infty} P_{\MM}\cdot u^n]_{\stackrel{D}{\sim}}| < \infty$;\label{item:FiniteDClass}
\item $[u]_{\stackrel{A}{\sim}}$ has finite order in $D_{\MM}$;\label{item:FiniteAOrder}
\item $[u]_{\stackrel{D}{\sim}}$ has finite order in $D'_{\MM}$.\label{item:FiniteDOrder}
\end{enumerate}

\end{lemma}

\begin{proof}
We will prove $\ref{item:FiniteOrbit} \Rightarrow \ref{item:BoundedAClass} \Rightarrow \ref{item:BoundedDClass} \Rightarrow \ref{item:FiniteDClass} \Rightarrow \ref{item:FiniteDOrder} \Rightarrow \ref{item:FiniteOrbit}$ and
$\ref{item:BoundedAClass} \Rightarrow \ref{item:FiniteAClass} \Rightarrow \ref{item:FiniteAOrder} \Rightarrow \ref{item:FiniteDOrder}$.

\begin{proof}[($\ref{item:FiniteOrbit} \Rightarrow \ref{item:BoundedAClass} \Rightarrow \ref{item:BoundedDClass}$)]\let\qed\relax
For all $n\in \NN$, we have $|P_{\MM}\cdot u^n|\leq |P_{\MM}\cdot u^\omega| < \infty$.
It's clear that  $|[P_{\MM}\cdot u^n ]_{\stackrel{D}{\sim}}| \leqslant |[P_{\MM}\cdot u^n ]_{\stackrel{A}{\sim}}| \leqslant
|P_{\MM}\cdot u^n |$.
\end{proof}

\begin{proof}[($\ref{item:BoundedAClass} \Rightarrow \ref{item:FiniteAClass}$)]\let\qed\relax
Consider $S = D_{\MM} \ZS Q^*$ and apply Lemma \ref{lemma:UnionOfBounded}
to $S$ with $E = Q$ and $X_i = [P_{\MM}\cdot u^i ]_{\stackrel{A}{\sim}}$.
We get that $|\bigcup_i[X_i]_{\stackrel{@}{\sim}}|$ is finite. As $D_{\MM}$ acts faithfully on $Q^*$, this implies that $|\bigcup_i[P_{\MM}\cdot u^i ]_{\stackrel{A}{\sim}}|<\infty$.
\end{proof}

\begin{proof}[($\ref{item:BoundedDClass} \Rightarrow \ref{item:FiniteDClass}$)]\let\qed\relax
From Lemma \ref{lemma:ActionOnDeltaPrime} it follows that we can consider $S = D'_{\MM} \ZS P_{\MM}$.
Apply Lemma \ref{lemma:UnionOfBounded}
to $S$ with $E = [Q]_{\stackrel{Q}{\sim}}$ and $X_i = [P_{\MM}\cdot u^i ]_{\stackrel{D}{\sim}}$.
\end{proof}

\begin{proof}[($\ref{item:FiniteAClass} \Rightarrow \ref{item:FiniteAOrder}$ and $\ref{item:FiniteDClass} \Rightarrow \ref{item:FiniteDOrder}$)]\let\qed\relax
Since $\{u^n\} \subseteq \bigcup_{n=0}^{\infty} P_{\MM}\cdot u^n$, 
we have 
\[
\{[u]^k_{\stackrel{A}{\sim}}\} \subseteq \bigcup_{n=0}^{\infty} [P_{\MM}\cdot u^n]_{\stackrel{A}{\sim}} \text{ and }
\{[u]^k_{\stackrel{D}{\sim}}\} \subseteq \bigcup_{n=0}^{\infty} [P_{\MM}\cdot u^n]_{\stackrel{D}{\sim}}.
\]
\end{proof}

\begin{proof}[($\ref{item:FiniteAOrder} \Rightarrow \ref{item:FiniteDOrder}$)]\let\qed\relax
This is clear, since $D'_{\MM}$ is a quotient of $D_{\MM}$.
\end{proof}

\begin{proof}[($\ref{item:FiniteDOrder} \Rightarrow \ref{item:FiniteOrbit}$)]\let\qed\relax
If $[u]_{\stackrel{D}{\sim}}$ generates a finite subsemigroup in $D'_{\MM}$, 
then there must exist $k,l\in \NN$ with $k<l$ and $[u^k]_{\stackrel{D}{\sim}} = [u^l]_{\stackrel{D}{\sim}}$. 
Then, we must have $\St_{P_{\MM}}(u^l) = \St_{P_{\MM}}(u^m)$ for all $m\geq l$. 
Indeed, suppose that there exist $m>l$ and $g\in \St_{P_{\MM}}(u^l)$ 
such that $g\cdot u^m \ne u^m$. Let us assume that $m$ is minimal for this property, 
meaning that $g\cdot u^n = u^n$ for all $n<m$. 
Therefore, $g\in \St_{P_{\MM}}(u^{m-1})$. Since $k<l\leq m-1$, we have 
$g\in \St_{P_{\MM}}(u^k)$ and $g@u^k \in \St_{P_{\MM}}(u^{m-1-k})$. 
By the assumption that $u^k \stackrel{D}{\sim} u^l$, we have $g@u^k = g@u^l$, 
which implies that $g\in \St_{P_{\MM}}(u^{l+m-1-k})$. Since $l-1-k \geq 0$, 
this means that $g\in \St_{P_{\MM}}(u^m)$, a contradiction.
Then $\St_{P_{\MM}}(u^{\omega}) = \St_{P_{\MM}}(u^{l})$ and 
\[|{P_{\MM}} \cdot u^{\infty}| = [{P_{\MM}}:\St_{P_{\MM}}(u^{\omega})] 
=[{P_{\MM}}:\St_{P_{\MM}}(u^{l})] = |{P_{\MM}} \cdot u^l| < \infty.
\]
\end{proof}

\end{proof}

\section{Free subsemigroups}\label{sec:FreeSubsemigroups}

\subsection{Changing the alphabet}

Let $\MM = (Q,A,\tau)$ be an invertible Mealy automaton and let $u \in A^*$, $|u| = k$.
We can construct a new Mealy automaton $\MM_k = (Q, A^k, \tau_k)$, where the map 
$\tau_k:Q \times A^k \to A^k \times Q$ is defined in natural way. We have $\Delta_{\MM_k} \leqslant \Delta_{\MM}$.

Suppose $[u]$ has infinite order in $D_{\MM}$, then from 
Lemma \ref{lemma:ClassesFiniteIFFOrbitsFinite}
it follows that $|P_{\MM} \cdot u^{\omega}| = |P_{\MM_k} \cdot u^{\omega}|=\infty$, and thus (again from \ref{lemma:ClassesFiniteIFFOrbitsFinite}) $[u]$ has infinite order in $D_{\MM_k}$.

Suppose there is a free non-commutative subsemigroup
in $D'_{\MM_k}$, then the corresponding elements of $D'_{\MM}$ generate a free subsemigroup in 
$D'_{\MM}$ and a free subsemigroup in $D_{\MM}$.

Hence, for theorem \ref{thm:main2}, 
it is enough to only consider the following case:

\begin{prop}\label{prop:CaseOfOneLetter}
Let $\MM = (Q,A,\tau)$ be a self-invertible Mealy automaton and let 
$|P_{\MM} \cdot a^{\omega}| = \infty$ 
where $a\in A$.
Then there exist $x, y \in A^*$ such that $[x]_{\stackrel{D}{\sim}}$ and $[y]_{\stackrel{D}{\sim}}$ freely generate a free subsemigroup in $D'_{\MM}$.
\end{prop}

\subsection{Proof of Proposition \ref{prop:CaseOfOneLetter}}

Apply to $\MM$ and $a$ Lemma \ref{lemma:SFT} to find numbers $T, d$ such that all words in $\Lan_{P_{\MM}}(a)$ with length $T$ have degrees equal to $d>1$.

We construct a directed labelled graph $\Gamma$. Its vertices are 
$\{ v \in \Lan_{P_{\MM}}(a) \mid 0 \leq |v| \leq T\}$.
Edges of $\Gamma$ will be labelled by letters of $A$. 
Let $v_2 = v_1a_i$ for some $a_i \in A$, then we draw an arrow $v_1 \to v_2$ 
with label $a_i$.
If $|v_1| = |v_2| = T$ then we draw an arrow $v_1 \to v_2$ if and only if there exist 
$a_i, a_j \in A$, $v' \in A^{T-1}$ such that $v_1 = a_i v'$, $v_2 = v'a_j$ 
and $a_i v' a_j \in \Lan_{P_{\MM}}(a)$. 
In this case, the edge $v_1 \to v_2$ is labelled by $a_j$.

Some properties of $\Gamma$:
\begin{enumerate}
\item the vertices of $\Gamma$ are divided into levels. 
The level $0$ consists of one vertex $V_{\varepsilon}$, where $\varepsilon\in A^*$ is the empty word.
For any $0 \leq k \leq T$ the $k$-th level consists of words of $\Lan_{P_{\MM}}(a)$ of length $k$.
\item $\Lan_{P_{\MM}}(a)$ is the set of words that are written along paths starting from $V_{\varepsilon}$.
\item For any path in $\Gamma$ the word that is written along it belongs to $\Lan_{P_{\MM}}(a)$.
\item Any path in $\Gamma$ with length  $\geq T$ ends at a vertex of level $T$.
\item For each vertex $V$ of level $T$, there exists a path from $V_{\epsilon}$ to $V$ of length $T$.
\item All vertices of level $T$ have outgoing degree equal to $d$. 
\end{enumerate}

Denote a path starting from $V\in \Gamma$ by ${}_Vu$, where $u$ is the word written along this path.

We review some basic facts about directed graphs.
We say that a subgraph of a directed graph is \emph{strongly connected} if for any two vertices, there is a directed path joining the first to the second.
A \emph{strongly connected component} of a directed graph is a strongly connected subgraph that is maximal for this property, meaning that no additional edges or vertices can be added to it without breaking its property of being strongly connected. 
The collection of strongly connected components forms a partition of the set of vertices of the graph.
If each strongly connected component is contracted to a single vertex, 
the resulting graph is a directed acyclic graph called the \emph{condensation} of the graph.

In the condensation of $\Gamma$ there exists a vertex $R$ without any outgoing edge.
In the corresponding strongly connected component $R$ of $\Gamma$ all the outgoing edges must therefore go to $R$.
In particular, all the vertices of $R$ must belong to level $T$.

\begin{lemma}\label{lemma:FindPath}
For any $T' \geqslant T$ and $V\in R$ there exists a word $u \in \Lan_{P_{\MM}}(a)$ 
of length $T'$ such that the path 
${}_{V_{\varepsilon}}u$ ends at $V$. 
\end{lemma}
\begin{proof}
Since $R$ is strongly connected, 
for any $V_i \in R$ there exists an incoming edge $V_j \to V_i$ for some $V_j\in R$.
Hence we can find in $R$ a path of length $T' - T$, ending at $V$.
We conclude using the fact that for any $V_i\in R$, there is a path of length $T$ from the root to $V_i$.
\end{proof}

Take an arbitrary vertex $V \in R$.
As every outgoing edge of $V$ leads to a vertex of $R$, we can find $d$ cycles ${}_Vc_1, \dots, {}_Vc_d $ starting and ending at $V$ 
and such that their first edges are all different (the cycles can pass through one edge multiple times).
Without loss of generality, we can assume that all ${}_Vc_i$ have the same length (replacing each cycle ${}_Vc_i$ 
by some power ${}_Vc_i^k$ if necessary).

We will prove Proposition \ref{prop:CaseOfOneLetter} by contradiction. Let us assume that every two-generated subsemigroup of $D'_{\MM}$ is not free.

\begin{lemma}\label{lemma:d_cycles}
If $D'_{\MM}$ contains no free two-generated subsemigroup, there exist cycles ${}_Vx, {}_Vw_1,\dots , {}_Vw_d$ starting and ending at $V$ such that 
$|w_1| = |w_2| = \dots = |w_d|$,
\begin{align*}
[xw_1]_{\stackrel{D}{\sim}} = [xw_2]_{\stackrel{D}{\sim}} = \dots = [xw_d]_{\stackrel{D}{\sim}} \text{    in  } D'_{\MM},
\end{align*}

and the cycles ${}_Vw_1, {}_Vw_2, \dots, {}_Vw_d$ begin with $d$ different edges.
\end{lemma}

\begin{proof}
We will prove by induction that for any  $1 \leqslant k  \leqslant d$ there exist 
$k+1$ cycles ${}_Vx, {}_Vw_{1},\dots, {}_Vw_{k}$
such that $[xw_{1}]_{\stackrel{D}{\sim}} = \dots = [xw_{k}]_{\stackrel{D}{\sim}}$ and ${}_Vc_i$ 
is a beginning of ${}_Vw_{i}$ for any $1 \leq i \leq k$. 

The case $k = 1$ is clear.
Let us now assume that it is true for some $k<d$ and let us prove that it then also holds for $k+1$.
Suppose that we have cycles ${}_Vx, {}_Vw_{1},\dots, {}_Vw_{k}$ satisfying the above conditions. 
Since ${}_Vw_k$ starts with $c_k$, we can write ${}_Vw_{k} = {}_Vc_kw'$, where ${}_Vw'$ 
is a cycle in $\Gamma$. 
By our assumption, the semigroup generated by $[xc_kw']$ and $[xc_{k+1}w']$ is not free.

This means that there are two different words $ W_1 , W_2 \in \{a,b\}^* $, 
such that  when all the letters $ a $ are replaced by $xc_kw'$, 
and all the letters $ b $ are replaced by $xc_{k+1}w'$ we obtain the same elements in $D'_{\MM}$. 
We denote by $\sim$ this congruence relation on $\{a,b\}^*$.

Without loss of generality, we can assume that $|W_1| = |W_2|$. 
Indeed, if $|W_1| < |W_2|$, then consider the words $W_1aW_2$ and $W_2aW_1$. We have that $W_1aW_2\sim W_2aW_1$. 
If these words are different as elements of $\{a,b\}^*$,  then everything is fine. 
Otherwise, the $(|W_1|+1)$th letter of $W_2$ is $a$, 
and then words $W_1bW_2\sim W_2bW_1$ are different as elements of $\{a,b\}^*$.

Let $W_3\in \{a,b\}^*$ be the largest common prefix of $W_1$ and $W_2$. As $W_1\ne W_2$ and $|W_1|=|W_2|$, we must have that $|W_3|<|W_1|$. Thus, without loss of generality, there exist (possibly empty) words $W_1', W_2' \in \{a,b\}^*$ such that $W_1=W_3aW_1'$ and $W_2=W_3bW_2'$.

Replacing in $W_1=W_3aW_1'$ and $W_2=W_3bW_2'$ all the letters $a$ by $xc_kw'$ and all the letters $b$ by $xc_{k+1}w'$, we obtain two words $yxc_kw't_1$ and $yxc_{k+1}w't_2$, where $|y|=|W_3||xc_kw'|$, $|t_1| = |t_2|$ and
$[yxc_kw't_1]_{\stackrel{D}{\sim}} = [yxc_{k+1}w't_2]_{\stackrel{D}{\sim}}$ in $D'_{\MM}$.
Note also that ${}_Vy$, ${}_Vx$, ${}_Vc_k$, ${}_Vc_{k+1}$, ${}_Vw'$, ${}_Vt_1$ and ${}_Vt_2$ 
are cycles in $\Gamma$. 

It is clear that
\begin{align*}
[yxw_{1}t_1]_{\stackrel{D}{\sim}} = [yxw_{2}t_1]_{\stackrel{D}{\sim}} = \dots = 
[yxc_kw't_1]_{\stackrel{D}{\sim}} = [yxc_{k+1}w't_2]_{\stackrel{D}{\sim}} \text{    in  } D'_{\MM}.
\end{align*}

Thus, taking the $k+2$ cycles 
$_{V}yx$, $_{V}w_1t_1$, $_{V}w_2t_1$, \dots, $_{V}w_kt_1$, $_{V}c_{k+1}w't_2$, we conclude that the result is also true for $k+1$. Therefore, by induction, it is true for $k=d$.
\end{proof}

Recall that all the vertices of the strongly connected component $R$ belong to level $T$. 
Suppose that there are $n$ vertices in $R$ and denote these vertices by $V_1,\dots, V_n$. 
Using Lemma \ref{lemma:d_cycles} for $V_1\in R$, we find $d+1$ cycles:
${}_{V_1}x, {}_{V_1}w_{1}, \dots {}_{V_1}w_{d}$.

Denote by $C$ the maximum of length of these cycles.

\begin{lemma}\label{lemma:changing}
Let $N > T + C$, $M > N + C$ and
let $v_1$, $v_2$ be words of length $M$ from $\Lan_{P_{\MM}}(a)$ such that $v_1$ and $v_2$ 
have common prefix of length $N$.
Then there exists $\widetilde{v_2} \in \Lan_{P_{\MM}}(a)$ such that $|\widetilde{v_2}| = M$, 
$\widetilde{v_2} \stackrel{D}{\sim} v_2$, and the words $v_1$ and $\widetilde{v_2}$ have common prefix of length $N+1$.
\end{lemma}

\begin{proof}

Suppose $v_1 = pu_1$, $v_2 = pu_2$, $|p| = N$. 
Lemma \ref{lemma:FindPath} gives us a word $t$ such that $|t| = N-|x|$, $t \in \Lan_{P_{\MM}}(a)$ 
and the path ${}_{V_{\varepsilon}}t$ ends at $V_1$.

There exists a path ${}_{V_{\varepsilon}}txw_1$, and this path ends at $V_1$.
We can prolong ${}_{V_{\varepsilon}}txw_1$ in arbitrary way and get a path 
${}_{V_{\varepsilon}}txw_1s$ of length $M$.

Since $txw_1s \text{ and }pu_2 = v_2 \in P_{\MM}\cdot a^{M}$, there exists $g\in P_{\MM}$ such that 
$g\cdot v_2 = txw_1s$.

The words $g\cdot v_2 = txw_1s$ and $g\cdot v_1$ have a common prefix $tx$ of length $N$, and 
$\Deg_{\Lan_{P_{\MM}}(a)}(tx) = d$. 
As the words $w_i$ each begin with a different letter, there must exist some $1\leq i \leq d$ such that the words $txw_is$ and $g\cdot v_1$ have a common prefix of length $(N+1)$.
Since $xw_{1} \stackrel{D}{\sim} xw_{i}$, then $txw_is \stackrel{D}{\sim} txw_1s$.

Let $\widetilde{v_2} = g^{-1}\cdot txw_is$. 
The words $\widetilde{v_2}$ and $v_1$ have common prefix of length $(N+1)$, 
and it follows from Lemma \ref{lemma:ActionOnDeltaPrime} that $\widetilde{v_2} \stackrel{D}{\sim} v_2$.

\end{proof}

\begin{lemma}\label{lemma:FewClasses}
For all $k\in \N$, we have $|[P_{\MM}\cdot a^k]_{\stackrel{D}{\sim}}|\leq |A|^{T + 2C +1}$.
\end{lemma}

\begin{proof}

For $k\leq T+2C+1$, the result is clear. Let us now assume that $k>T+2C+1$.

Let $P_{\MM}\cdot a^{T+C+1} = \{v_1,v_2,\dots, v_K\}$ and let us choose $K$ words $w_1,w_2, \dots, w_K \in P_{\MM}\cdot a^K$ such that for all $i$, $v_i$ is a prefix of $w_i$. Notice that $K\leq |A|^{T+C+1}$.

Let $u\in P_{\MM}\cdot a^k$ be an arbitrary element. By Lemma \ref{lemma:FactorLanguage}, there exists $i$ such that $u$ and $w_i$ have the same prefix of length $T+C+1$. As $|u|=k>T+2C+1$, we have $k-(T+C+1)>C$, so we can apply Lemma \ref{lemma:changing} to find an element $u_1\in P_{\MM}\cdot a^k$ such that $u_1\stackrel{D}{\sim} u$ and $u_1$ shares with $w_i$ a prefix of length $T+C+2$. If $k-(T+C+2)>C$, we can apply Lemma \ref{lemma:changing} again to find an element $u_2$ sharing a prefix of length $T+C+3$ with $w_i$ and such that $u_2\stackrel{D}{\sim} u_1\stackrel{D}{\sim} u$. By repeating this process, we can find an element $u'\in P_{\MM}\cdot a^k$ having a common prefix of length $k-C$ with $w_i$ and such that $[u']_{\stackrel{D}{\sim}} =[u]_{\stackrel{D}{\sim}}$ in $D'_{\MM}$.

Therefore, there are at most $K|A|^C$ different elements in $[P\cdot a^k]_{\stackrel{D}{\sim}}$. As $K\leq |A|^{T+C+1}$, the result follows.

%Consider all words of $\Lan_{P_{\MM}}(a)$ with length $T + C_1 + 1$:
%\[
%v_1, v_2, \dots, v_N; \: \:{N \leq |A|^{T + C_1 + 1}}.
%\]\todo{What is $C_1$?}

%Choose in $P_{\MM} \cdot a^k$ words $w_1, \dots, w_N$ such that $v_i$ is a beginning of $w_i$ for any $i$.

%Take arbitrary $u \in P_{\MM}\cdot a^k$. There exists $i$ such that $w_i$ and $u$ 
%have common prefix of length $T + C + 1$.
%Applying lemma \ref{lemma:changing}, we find $u_1\in \Lan_{P_{\MM}}(a)$ such that $u_1 \stackrel{D}{\sim} u$ 
%and the words $u_1$, $w_i$ have a common prefix of length $T + C + 2$. 
%We can apply this lemma again and find $u_2 \in \Lan_{P_{\MM}}(a)$ such that $u_2 \stackrel{D}{\sim} u_1$ 
%and words  $u_2$ and $w_i$ have a common prefix of length $T + C + 3$, and so on.
%In this manner we find a word $u'$ such that $[u']_{\stackrel{D}{\sim}} = [u]_{\stackrel{D}{\sim}}$ in 
%$D'_{\MM}$, $u'$ and $w_i$ have a common prefix of length 
%$k - C_2 - 1$.

%Note that starting with different $u$ we obtain not more then $N \cdot |A|^{C_2 + 1}$ different words $u'$.
\end{proof}

By Lemmas \ref{lemma:FewClasses} and\ref{lemma:ClassesFiniteIFFOrbitsFinite}, we must have $|P_{\MM}\cdot a^\omega| < \infty$, which is a contradiction.
Thus, we conclude that $D'_{\MM}$ must contain a free subsemigroup on two generators.
This proves Proposition \ref{prop:CaseOfOneLetter}. \qed

%In assumption that $D'_{\MM}$ does not contain a two-generated free subsemigroup, 
%Lemmas \ref{lemma:FewClasses} and \ref{lemma:ClassesFiniteIFFOrbitsFinite} give us a contradiction, that 
%proves Proposition \ref{prop:CaseOfOneLetter}.

\subsection{Proofs of Theorem \ref{thm:MainTheorem} and Corollary \ref{cor:ExponentialGrowth}}
\subsubsection*{Proof of Theorem \ref{thm:MainTheorem}}
Let $\MM=(Q,A,\tau)$ be an invertible Mealy automaton.
If every element of $D_{\MM}$ is of finite order, then $D_{\MM}$ clearly contains no free subsemigroup and so neither does $D'_{\MM}$.

On the other hand, if there exists $x\in D_{\MM}$ of infinite order, then by Lemma \ref{lemma:ClassesFiniteIFFOrbitsFinite}, $|P_{\MM}\cdot x^\omega|=\infty$.
Thus, it follows from Proposition \ref{prop:CaseOfOneLetter} and the discussion preceding it that there exist $y,z\in D'_{\MM}$ that freely generate a free subsemigroup in $D'_{\MM}$.

This proves Theorem \ref{thm:main2} and therefore Theorem \ref{thm:mainDual}.
Hence, by Remark \ref{rem:MainThmEquivalentMainDual}, Theorem \ref{thm:MainTheorem} is proved.

\subsubsection*{Proof of Corollary \ref{cor:ExponentialGrowth}}
Let $\MM=(Q,A,\tau)$ be an invertible and reversible Mealy automaton and let $G=P_{\widetilde{\MM}}$ be the group generated by $\MM$.
Let us suppose that $G$ contains an element of infinite order.
Then, so does $P_{\MM}$, by Proposition \ref{prop:InfiniteOrderInSemigroupSameAsGroup}.
%Indeed, if every element of $P_{\MM}$ is of finite order, then the inverse of any element of $P_{\MM}$ is also an element of $P_{\MM}$, which means that $P_{\MM}$ is a group.
%Therefore, from Corollary \ref{cor:IfAlreadyGroupThenNoChange}, we get that $P_{\MM} = G$.
%We conclude that $G$ is torsion, which is absurd.

%Thus, if $G$ admits an element of infinite order, so does $P_{\MM}$.
By Theorem \ref{thm:MainTheorem}, we conclude that $P_{\MM}$ contains a subsemigroup of rank two.
Therefore, $G$ contains a subsemigroup of rank two and is thus of exponential growth.

To prove that no infinite virtually nilpotent group can be generated by an invertible and reversible automaton, it suffices to remark that a finitely generated torsion virtually nilpotent group is finite (see for example \cite{DK18}, Proposition 13.65). Thus, an infinite virtually nilpotent group must contain an element of infinite order. As such a group is of polynomial growth by a theorem of Wolf \cite{Wolf68}, the result follows.

\section{Appendix. Orbits and languages.}\label{section:languages}

\subsection{Regularity}

It follows from Lemma \ref{lemma:SFT} that for any invertible Mealy automaton $\MM = (Q, A, \tau)$ 
and any periodic sequence $u^{\omega} \in A^{\omega}$, the set of prefixes of its orbit $P_{\MM} \cdot u^{\omega}$ forms a regular language.

What happens if we consider some subgroup of $P_{\MM}$ instead of whole $P_{\MM}$, 
or a pre-periodic word instead of a periodic one, or a non-invertible Mealy automaton?

\begin{example}\label{example:NotRegular}
There exists an invertible Mealy automaton 
$\MN = (Q_{\MN},A_{\MN}, \tau_{\MN})$, $a\in A_{\MN}$ and $\psi\in Q_{\MN}^*$ 
such that the set of all prefixes of words in $\langle \psi \rangle \cdot a^{\omega}$ 
is not a regular language.

Indeed, let
$Q_{\MN}=\{\epsilon, s,t, x,y\}$ and
\[A_{\MN} = \{a, b, c, d, e, f, g, h, i, j, k\}.\]

The structure of the automaton is given by its map
$\tau\colon Q_{\MN}\times A_{\MN} \to A_{\MN}\times Q_{\MN}$.

The state $\epsilon$ is the identity, and
$\tau (\epsilon, a) = (a, \epsilon)$ for all $a\in A_{\MN}$.

\[\begin{array}{cr|cccc|}
& & \multicolumn{4}{c|}{\text{in state}}\\
& & s & t & x & y \\ \hline
\multirow{11}{*}{\rotatebox{90}{letter}} 
& a & (b, \epsilon) & (j, \epsilon) & (a, \epsilon) & (a, \epsilon)\\
& b & (e, s) & (b, \epsilon)& (c, \epsilon) & (b, \epsilon)\\
& c & (f, t) & (c, \epsilon) & (b, \epsilon) & (c, \epsilon)\\
& d & (a, \epsilon) & (d, \epsilon) & (d, \epsilon) & (d, y)\\
& e & (d, \epsilon) & (e, \epsilon) & (f, \epsilon) & (e, y)\\
& f & (c, \epsilon) & (f, \epsilon) & (e, x) & (f, y)\\
& g & (g, \epsilon) & (h, \epsilon) & (g, x) & (g, \epsilon)\\
& h & (h, \epsilon) & (j, \epsilon) & (h, x) & (i, \epsilon)\\
& i & (i, \epsilon) & (k, \epsilon) & (i, x) & (h, \epsilon)\\
& j & (j, \epsilon) & (a, t) & (j, \epsilon) & (k, \epsilon)\\
& k & (k, s) & (i, \epsilon) & (k, \epsilon) & (k, \epsilon)\\
\hline
\end{array}\]

We claim that
\begin{enumerate}
\item $x$ and $y$ commute in $P_{\MN}$;
\item if $m > 0$, then $y^{2^n}x^{2^m}s$ acts on $A_{\MN}$ as $(a, b, e, d)(c, f)$ and 
$(y^{2^n}x^{2^m}s)^4@a = y^{2^{n+1}}x^{2^{m-1}}s$ in $P_{\MN}$;
\item $y^{2^n}xs$ acts on $A_{\MN}$ as $(a, c, e, d)(b, f)$
and $(y^{2^n}xs)^4@a = y^{2^{n+1}}xt$ in $P_{\MN}$;
\item if $n > 0$, then  $y^{2^n}x^{2^m}t$ acts on $A_{\MN}$ as $(a,j,h,g)(i,k)$ 
and $(y^{2^n}x^{2^m}t)^4@a = y^{2^{n-1}}x^{2^{m+1}}t$ in $P_{\MN}$;
\item $yx^{2^m}t$ acts on $A_{\MN}$ as $(a,k,h,g)(i,j)$
and $(yx^{2^m}t)^4@a = yx^{2^{m+1}}s$ in $P_{\MN}$.
\end{enumerate}

To see that, we draw the dual Moore diagram of our transducer (see Figure \ref{Fig:DualMooreAutomatonOfExample}). 
Several loops with empty outputs are not shown.

\begin{figure}[h]
\begin{center}
\begin{dualmoore}[scale=1]
\node[state] (a) at (0, 0) {$a$};
\node[state] (b) at (2, 0) {$b$};
\node[state] (c) at (4, 0) {$c$};
\node[state] (d) at (0, -2) {$d$};
\node[state] (e) at (2, -2) {$e$};
\node[state] (f) at (4, -2) {$f$};
\node[state] (g) at (0, 2) {$g$};
\node[state] (h) at (-2, 2) {$h$};
\node[state] (i) at (-4, 2) {$i$};
\node[state] (j) at (-2, 0) {$j$};
\node[state] (k) at (-4, 0) {$k$};

\path (a) edge  [in=80,out=20, loop] node[above] {\:\:$(y,\epsilon)$} ()
edge [in=250, out=210, loop ] node[left] {$(x,\epsilon)$} ()
(a) edge node[above] {$(s, \epsilon)$} (b)
(b) edge [<->] node[above] {$(x, \epsilon)$} (c)
(b) edge[loop above] node[above] {$(y,\epsilon)$} ()
(c) edge[loop above] node[above] {$(y,\epsilon)$} ()
(d) edge[loop left] node[left] {$(y,y)$} ()
(d) edge[loop below] node[below] {$(x,\epsilon)$} ()
(e) edge[loop below] node[below] {$(y,y)$} ()
(f) edge[loop below] node[below] {$(y,y)$} ()
(c) edge [<->] node[right] {$(s,t)$} (f)
(f) edge[bend right] node[above] {$(x,x)$} (e)
(e) edge[bend right] node[below] {$(x,\epsilon)$} (f)
(e) edge node[below] {$(s,\epsilon)$} (d)
(d) edge node[left] {$(s,\epsilon)$} (a)
(b) edge node[right] {$(s,s)$} (e);

\path (a) edge node[below] {$(t, \epsilon)$} (j)
(j) edge [<->] node[below] {$(y, \epsilon)$} (k)
(j) edge[loop below] node[below] {$(x,\epsilon)$} ()
(k) edge[loop below] node[below] {$(x,\epsilon)$} ()
(g) edge[loop right] node[right] {$(x,x)$} ()
(g) edge[loop above] node[above] {$(y,\epsilon)$} ()
(h) edge[loop above] node[above] {$(x,x)$} ()
(i) edge[loop above] node[above] {$(x,x)$} ()
(k) edge [<->] node[left] {$(t,s)$} (i)
(i) edge[bend right] node[below] {$(y,y)$} (h)
(h) edge[bend right] node[above] {$(y,\epsilon)$} (i)
(h) edge node[above] {$(t,\epsilon)$} (g)
(g) edge node[left] {$(t,\epsilon)$} (a)
(j) edge node[left] {$(t,t)$} (h);

\end{dualmoore}
\end{center}
\caption{The dual Moore diagram of $\MN$.}\label{Fig:DualMooreAutomatonOfExample}
\end{figure}

Consider the element $\psi:= yxs \in Q^*_{\MN}$. We will show that the set 
\[L:= \bigcup_{n,m \in \NN} \psi^n \cdot a^m
\]

is not a regular language.
Assume the contrary and consider the following sequence:
\[\psi_0 := yxs, \psi_1 := y^2xt, \psi_2 := yx^2t, \psi_3:= yx^4s, 
\psi_4:=y^2x^2s, \psi_5:=y^4xs,  \dots
\]

\begin{equation*}
\psi_{i+1} = 
 \begin{cases}
   y^{2^{n+1}}x^{2^{m-1}}s &\text{if $\psi_i = y^{2^n}x^{2^m}s, m>0$}\\
   y^{2^{n+1}}xt &\text{if $\psi_i = y^{2^n}xs$}\\
   y^{2^{n-1}}x^{2^{m+1}}t &\text{if $\psi_i = y^{2^n}x^{2^m}t, n>0$}\\
   yx^{2^{m+1}}s & \text{if $\psi_i = yx^{2^m}t$}\\
 \end{cases}
\end{equation*}

\begin{remark}
This sequence of group elements corresponds to sequence of states $(q_i, m_i, n_i)$ of some Minsky machine.
This sequence can be presented as a path along the plane (see Figure \ref{Fig:PathMinskyMachine}).

For more about Minsky machines and automata groups, see \cite{BM}.

\begin{figure}[h]
\begin{tikzpicture}[node distance=2cm]

\node (0) at (0, 0) {$s$};
\node (1) at (0, 1) {$t$};
\node (2) at (1, 0) {$t$};
\node (3) at (2, 0) {$s$};
\node (4) at (1, 1) {$s$};
\node (5) at (0, 2) {$s$};
\node (6) at (0, 3) {$t$};
\node (7) at (1, 2) {$t$};
\node (8) at (2, 1) {$t$};
\node (9) at (3, 0) {$t$};
\node (10) at (4, 0) {$s$};
\node (11) at (3, 1) {$s$};
\node (12) at (2, 2) {$\circ$};
\node (13) at (2, 3) {$\circ$};
\node (14) at (3, 2) {$\circ$};
\node (15) at (4, 2) {$\circ$};
\node (16) at (1, 3) {$\circ$};
\node (17) at (5, 0) {$\circ$};
\node (18) at (5, 1) {$\circ$};
\node (18) at (4, 1) {$\circ$};
\draw[->] (0) edge (1) (1) edge (2) (2) edge (3)
(3) edge (4) (4) edge (5) (5) edge (6)
(6) edge (7) (7) edge (8) (8) edge (9)
(9) edge (10) (10) edge (11);

\end{tikzpicture}
\caption{The sequence of elements $\psi_i$, where only the last letter of each $\psi_i$ is shown.}
\label{Fig:PathMinskyMachine}
\end{figure}

\end{remark}

By induction we show that $\psi^N \cdot a^{\omega}$ starts with $a^n$ if and only if $4^n\divides N$ and that $\psi^{4^n}@a^n = \psi_{n}$.
It means that $a^kb \in L$ if and only if $\psi_k$ is of the form $y^{2^n}x^{2^m}s$ 
(not of form $y^{2^n}x^{2^m}t$), i.e. 
$2N^2 + N \leq k \leq 2N^2 + 3N$ for some $N\in \NN$.

We recall the well-known

\begin{lemma}[Pumping lemma]
Let $L$ be a regular language. Then there exists an integer $p \geq 1$ 
 such that every string $w \in L$  of length at least $p$ 
 can be written as $w = x y z$, 
such that $|y| \geq 1$, $|xy| < p$ and $xy^nz \in L$ for any $n$.
\end{lemma}

Applying this lemma, we get $k_1, k_2 \in \NN$ such that $a^{k_1+ nk_2}b \in L$ for any $n \in \NN$.
But for $N$  big enough  we can find $k'$ such that $2N^2 + 3N < k_1 + k'k_2 < 2(N+1)^2 + (N+1)$, so 
$a^{k_1 + k'k_2}b \not \in L$ and we have a contradiction with the regularity of $L$.

\end{example}

\begin{remark}
If we replace $Q_{\MN}$ by $Q_{\MN}^3$, we will get $\psi \in Q_{\MN}$.
From now on, by $Q_{\MN}$, we will actually mean $Q_{\MN}^3$.
\end{remark}

\begin{example}
There exist an invertible Mealy automaton $\MM_1 = (Q_1,A_1,\tau_1)$ and 
a pre-periodic sequence $vu^{\omega} \in A_1^*$ such that the set of prefixes of 
$Q^* \cdot vu^{\omega}$ is not a regular language.

Indeed, let $\MN = (Q_{\MN}, A_{\MN}, \tau_{\MN})$ 
and $\psi \in Q_{\MN}$ be as in Example \ref{example:NotRegular}. Then, the set $\bigcup \langle \psi \rangle \cdot a^n$ is not a regular language.

Let $Q_1 := Q_{\MN}$, $A_1 := A_{\MN} \cup \{a_1\}$, $\tau_{1}(q, \alpha) = \tau_{\MN}(q, \alpha)$ 
for $\alpha \in A_{\MN}$
and $\tau_1(q, a_1) = (a_1, \psi)$ for any $q\in Q_1$.

Then for any $g \in Q_1^*$ we have $g \cdot a_1a^{n} = a_1 (\psi^{|g|}\cdot a^n)$, 
and the set of all prefixes of
$Q_1^* \cdot a_1a^{\omega}$ is not regular.

\end{example}

\begin{example}
There exist a non-invertible Mealy automaton $M_2 = (Q_2,A_2, \tau_2)$ and 
a periodic sequence $u^{\omega}\in A_2^{\omega}$  such that the set of all prefixes of  
$Q_2^* \cdot u^{\omega}$ is not a regular language.

Indeed, let $\MN = (Q_{\MN},A_{\MN}, \tau_{\MN})$, $a\in A_{\MN}$ and $\psi\in Q_{\MN}$ 
be as in Example \ref{example:NotRegular}, 
so that the set of all prefixes of  $\psi^* \cdot a^{\omega}$ is not a regular language.
Let $Q_2 := Q_{\MN}\cup \{r_1, r_2, r_3\}$, 
$A_2 = A_{\MN} \cup \{0\}$, 
$\tau_2(q, \alpha) = \tau_{\MN}(q, \alpha)$ for $\alpha\in A_{\MN}$, $q \in Q_{\MN}$.
For any $a_i\in A_{\MN}$, let
$\tau_2(r_1, a_i) = (0, r_2)$, $\tau_2(r_2, a_i) = (a, r_2)$, $\tau_2(r_3, a_i) 
= (a, \epsilon)$.
Finally, let $\tau_2(q, 0) = (0, \psi)$ for any $q\in Q_2$.

Then $(Q \cup \{r_2, r_3\})^* \cdot a^{\omega} \subseteq A_{\MN}^{\omega}$ and 
$r_1 \cdot A_{\MN}^{\omega} = 0a^{\omega}$. 
For any $g \in Q_2^*$ we have $g \cdot 0 a^{\omega} = 0( \psi^{|g|}\cdot a^{\omega})$.
This means that the set of all prefixes of  $Q_2^* \cdot a^{\omega}$ is not a regular language.

\end{example}

But not all the results are negative.

\begin{prop}
Let $\MM = (Q, A, \tau)$ be a bi-reversible Mealy automaton and let $u, v \in A^*$. 
Then the set of all prefixes of $Q^* \cdot uv^{\omega}$ forms a regular language.
\end{prop}

\begin{proof}
It is enough to consider the case where $u, v \in A$.
We may also assume that for any $a \in A^*$, there is a formal inverse $a^{-1} \in A$ 
such that $(q@a)@a^{-1} = q$ and $(q@a)\cdot a^{-1} = (q \cdot a)^{-1}$ for all $q\in Q$.

Suppose that $g\cdot u = u', g@u = g'$ for some $g\in Q^*$.
Then, $g \cdot uv^{\omega} = u' g'\cdot v^{\omega}$. 
We know that $g'\cdot u^{-1} = (g@u) \cdot u^{-1} = (g \cdot u)^{-1} = u'^{-1}$,
so $g \cdot uv^{\omega} = (g'\cdot u^{-1})^{-1} g'\cdot v^{\omega}$.
Since $\MM$ is reversible, $Q^* \cdot uv^{\omega} = \{(g\cdot u^{-1})^{-1} g\cdot v^{\omega}\mid g\in Q^*\}$.

We know that the set of prefixes of $Q^* \cdot v^{\omega}$ forms a regular language.

There is a homomorphism $p: Q^* \to \Sym(A)$ given by $p(g)(a) = g\cdot a$.
We have a sequence of nested finite subgroups 
\[\Sym (A) \geq p(\St_{Q^*}(v)) \geq p(\St_{Q^*}(v^2)) 
\geq  p(\St_{Q^*}(v^3)) \geq \dots.
\]
There must exist some $n_0 \in \NN$ such that $p(\St_{Q^*}(v^n)) = p(\St_{Q^*}(v^{n_0}))$ for all $n \geq n_0$.

\begin{lemma}\label{lemma:OrbitsPreperiodicBireversible}
For all $n \geq n_0$, $a\in A$ and $w\in A^n$, the word $aw \in Q^* \cdot uv^n$ if and only if
$w \in Q^* \cdot v^n$ and some prefix of $aw$ belongs to $Q^* \cdot uv^{n_0}$.
\end{lemma}

\begin{proof}
$(\Rightarrow)$ Trivial.

$(\Leftarrow)$ Consider two subsets of $\Sym(A)$:
\[ S_1 := \{p(g)\mid g\in Q^*, g\cdot v^n = w\} \text{ and } S_2 := \{p(g)\mid g\in Q^*, 
g\cdot v^{n_0} = w_{1,\dots,n_0}\}.
\]

By assumption, these sets are not empty, and it is obvious that $S_1 \subseteq S_2$. On the other hand, 
let $g_0 \in Q^*$ be such that $w = g_0 \cdot v^n$. Then,

\[ S_1 = p\left(g_0 \St_{Q^*} (v^n) \right) =  p\left(g_0 \St_{Q^*} (v^{n_0})\right) = S_2. \]

By assumption, there exists $g_1\in Q^*$ such that $g_1\cdot uv^{n_0} = aw_{1\dots n_0}$.
Thus, $(g_1@u)\cdot u^{-1} = a^{-1}$ and $(g_1@u)\cdot v^{n_0} = w_{1\dots n_0}$.
Therefore, $p(g_1@u) \in S_2$.
As $S_2 = S_1$, there exists $g_2\in Q^*$ such that $g_2\cdot v^n = w$ and $p(g_2)=p(g_1@u)$.
Let $g_3 = g_2@u^{-1}$.
We have
\[g_3\cdot uv^n = (g_2\cdot u^{-1})^{-1} g_2\cdot v^n = aw.\]
Hence, $aw\in Q^*\cdot uv^n$.
%\[aw \in Q^* \cdot uv^n \Leftrightarrow a^{-1}\in S_1(u^{-1})\Leftrightarrow 
%a^{-1}\in S_2(u^{-1}) \Leftrightarrow 
%aw_{1\dots n_0} \in Q^* \cdot uv^{n_0}.
%\]
   
\end{proof}

By Lemma \ref{lemma:OrbitsPreperiodicBireversible}, the set of all prefixes of $Q^*\cdot uv^\omega$ can be expressed as the intersection of two regular languages. It therefore forms a regular language.

\end{proof}

\subsection{Action on periodic sequences}
\def\state#1#2{\text{\uppercase\expandafter{\romannumeral #1}}_{#2}}

\begin{example}\label{example:PeriodicWordsHaveFiniteOrbits}

We will show that there exist an invertible Mealy automaton $\MM = (Q, A, \tau)$ and an element $\psi\in P_{\MM}$ such that
\begin{enumerate}
\item $s$ has infinite order in $P_{\MM}$;
\item for any $u \in A^*$ the orbit $\langle s \rangle \cdot u^{\omega}$ is finite.
\end{enumerate}

As in Example \ref{example:NotRegular}, our construction is based on the realisation of Minsky machines as automata groups from \cite{BM}.

Consider a Minsky machine $M$ with a set of instructions
$S = \{a_1, a_2, b_1, b_2, c_1, c_2, d_1, d_2\}$.

The initial state of $M$ is $(a_1, 0, 0)$, 
and each next state is determined by the previous one by the following rules:

\begin{itemize}
\item[] $(a_1, m, n) \mapsto (b_1, m+1, n+1)$;
\item[] $(a_2, m, n) \mapsto (b_2, m+1, n+1)$;
\item[] $(b_1, m, n) \mapsto (d_1, m-1, n)$;
\item[] $(b_2, m, n) \mapsto (d_2, m, n-1)$;
\item[] $(c_1, m, n) \mapsto (a_1, m-1, n)$;
\item[]  $(c_2, m, n) \mapsto (a_2, m, n - 1)$;
\item[] $(d_1, m, n) \mapsto (m = 0 ? a_2 : c_1, m, n)$;
\item[] $(d_2, m, n) \mapsto (n = 0 ? a_1 : c_2, m, n)$.
\end{itemize}

There is no final state. In terms of \cite{BM}, instructions $a_1$ and $a_2$ are of type III, 
$b_1$ and $c_1$ are of type IV, $b_2$ and $c_2$ are of type V, $d_1$ has type VII and $d_2$ has type VIII.

The first states of a working machine are:
\begin{align*}
(a_1, 0, 0) \mapsto (b_1, 1, 1) \mapsto (d_1, 0, 1) \mapsto  (a_2, 0, 1) \mapsto (b_2, 1, 2) \mapsto  \\
\mapsto (d_2, 1, 1) \mapsto (c_2, 1, 1) \mapsto (a_2, 1, 0) \mapsto (b_2, 2, 1) \mapsto (d_2, 2, 0) \mapsto \\
\mapsto (a_1, 2, 0) \mapsto (b_1, 3, 1) \mapsto (d_1, 2, 1) \mapsto (c_1, 2, 1) \mapsto (a_1, 1, 1) \mapsto 
(b_1, 2, 2) \mapsto \dots
\end{align*}

It's easy to see that this machine works in a similar way to the machine from example \ref{example:NotRegular},
and the sequence of its instruction types $\state3{}, \state4{}, \state7{}, \state3{}, \state5{},\dots$ 
is not periodic.

We consider the automaton $\MM$ with stateset $Q = S^{\pm1} \sqcup \{\varepsilon, x, x^{-1}, y, y^{-1}\}$
and alphabet 
\[A=\{\state3i,\state4i,\state5i,\state7j,\state8j\mid i=1,2,\overline1,
\overline2;\;j=1,\dots,4,\overline1,\dots,\overline4\}.
\]

\noindent The state $\varepsilon$ is the identity, and 
$\tau(\varepsilon, a) = (a, \varepsilon)$ for all $a\in A$. 
All ``formal inverses'' will be inverse elements in $P_{\MM}$.

\noindent We define $\tau$ in the same way as in section 3.1 of \cite{BM}. 
For example, for  elements $a_1, a_2$ of type III and for all 
$t\in S\setminus \{a_1, a_2\}$, we define $\tau$ as 
\[\begin{array}{cr|cccc|}
 & & \multicolumn{4}{c|}{\text{input letter}}\\
 & & \state31 & \state32 & \state3{\overline1} & \state3{\overline2}\\ \hline
\multirow{5}{*}{\rotatebox{90}{in state}} & x & (\state31, x) & (\state32, x) & (\state3{\overline1}, x^{-1}) & (\state3{\overline2}, x^{-1})\\
 & y & (\state31, y) & (\state32, y) & (\state3{\overline1}, y^{-1}) & (\state3{\overline1}, y^{-1})\\
 & a_1 & (\state32, b_1) & (\state31, \epsilon) & (\state3{\overline2}, \epsilon) & (\state3{\overline1}, (b_1)^{-1})\\
 & a_2 & (\state32, b_2) & (\state31, \epsilon) & (\state3{\overline2}, \epsilon) & (\state3{\overline1}, (b_2)^{-1})\\
 & t & (\state3{\overline1}, \epsilon) & (\state3{\overline2}, \epsilon) & (\state31, \epsilon) & (\state32, \epsilon)\\ \hline
\end{array}
\]

Since $M$ does not stop, the element $\psi := xya_1$ has infinite order in $P_{\MM}$.

\noindent We construct a labelled, directed graph $\Gamma$, whose vertices are elements of $P_{\MM}$.
For $g\in P_{\MM}$ and $a\in A$, consider the minimal $d(g, a)$ such that $g^{d(g,a)} \cdot a = a$.
In our graph we put an edge from $g$ to $g^{d(g, a)} @ a$ with label $(a, d(g, a))$.

For a sequence $w = a_1a_2\dots \in A^{\omega}$ we find a path starting at $xya_1$
whose edges have first labels $a_1, a_2, \dots$. 
The product of second labels of its edges is equal to the size of the orbit of $w$.

Suppose that in $M$ there is a step $(s, m, n) \mapsto (s', m', n')$. 
Consider $g = (y^{2^n}x^{2^m}s)^{\pm h}$, where $h\in \langle x, y \rangle$. 
In \cite{BM} it is calculated that  all the edges from $g$ go to $\varepsilon$ or to elements of form 
$(y^{2^{n'}}x^{2^{m'}}s')^{\pm h'}$, where $h'\in \langle x, y \rangle$; 
and all the edges that do not go to $\varepsilon$ have labels that correspond to $s$. 
E.g., an edge from $y^{2^n}x^{2^m}d_1$ either has its first label of type $\state7j$ or leads to $\varepsilon$.

Since the sequence of instructions of working $M$ is not periodic, any periodic sequence of labels in $\Gamma$
lead to $\varepsilon$ after finite number of steps, and the orbit $\langle \psi \rangle \cdot u^{\omega}$ 
of any periodic sequence $u^\omega \in A^{\omega}$ is finite. 

\end{example}

\section{Outlook}\label{sec:Outlook}

We proved that if $\MM$ is a reversible Mealy automaton, the semigroup $P_{\MM}$ generated by $\MM$ is either periodic or contains a free subsemigroup of rank two.
Obviously, if $P_{\MM}$ is finite, then it is periodic, but it is natural to ask if the converse holds.
More precisely, does there exist a reversible Mealy automaton that generates an infinite periodic semigroup?

To the best of our knowledge, the answer to this question is not known, but there are some partial results in the case of invertible and reversible automata.
In this case, by Proposition \ref{prop:InfiniteOrderInSemigroupSameAsGroup}, the question reduces to the Burnside problem (i.e. the question of the existence of an infinite finitely generated periodic group) for the class of groups generated by invertible and reversible automata.
It was shown in \cite{GK15} and in \cite{DR16} that an infinite group generated by an invertible and reversible but not bireversible automaton must contain an element of infinite order.
For bireversible automata, it was proved by Klimann in \cite{Kli16} that a group generated by a 2-state bireversible Mealy automaton cannot be infinite and torsion. Klimann, Picantin and Savchuk later proved in \cite{KPS18} that the same result holds for groups generated by a connected 3-state bireversible automaton.

%it was proved by Godin and Klimann in \cite{GK18} and by D'Angeli and Rodaro in \cite{DR16} that a connected bireversible Mealy automaton with a prime number of states cannot generate an infinite torsion group.

%We proved that for an invertible automaton $\MM$, its dual semigroup $D_{\MM}$ is either periodic or contains a free subsemigroup of rank two.
%The semigroups $D_{\MM}$ and $P_{\MM}$ are both finite or infinite, and if they are finite, then all elements of $D_{\MM}$ obviously have finite order.

%But does there exist an invertible automaton with infinite, but periodic dual group? 
%We do not know. 
%A dual question: does there exist an infinite periodic semigroup, generated by a reversible Mealy automaton?
%Godin and Thibault proved in \cite{GK}, that 
%no connected reversible-invertible Mealy automaton of prime size can generate an infinite Burnside group.

By Lemma \ref{lemma:ClassesFiniteIFFOrbitsFinite} and Theorem \ref{thm:BothSemigroupsAreFinite}, we can reformulate our question as follows: 
does there exist an invertible automaton $\MM$ such that the group $P_{\MM}$ is infinite, but 
any periodic sequence $u^{\omega}$ has finite orbit?

There are also some partial results from this point of view.
Example \ref{example:PeriodicWordsHaveFiniteOrbits} shows that this is possible if we consider an infinite subgroup of an automaton group.
On the other hand, D'Angeli, Francoeur, Rodaro and W\"{a}chter showed in \cite{DFRW18} that for any finitely generated subgroup of an automaton group, there is a sequence (not necessarily periodic) with an infinite orbit.

\subsection{Algorithmic questions.}

We know that for an invertible automaton $\MM = (Q, A, \tau)$ and for some $u\in A^*$, all prefixes of $P_{\MM} \cdot u^{\omega}$ form a regular language, but we do not know of an effective way to describe it. 
We do not even know if there exists an algorithm that, for a given invertible automaton $\MM = (Q, A, \tau)$ and a given word $u\in A^*$, determines whether or not the orbit $P_{\MM} \cdot u^{\infty}$ is finite.

Note that for a cyclic subgroup of $P_{\MM}$ this problem is undecidable \cite{BM}.

%\begin{bibsection}
%\begin{biblist}
%\bibselect{math}
%\end{biblist}
%\end{bibsection}
\bibliography{biblio}{}
\bibliographystyle{plain}

\end{document}